\documentclass[]{interact}

\usepackage{epstopdf}
\usepackage[caption=false]{subfig}

\usepackage{amsfonts,mathtools}
\usepackage{xparse}
\usepackage{color}
\usepackage{mathrsfs}

\usepackage[natbibapa,nodoi]{apacite}
\theoremstyle{plain}
\newtheorem{theorem}{Theorem}[section]

\newtheorem{corollary}[theorem]{Corollary}

\theoremstyle{definition}

\newtheorem{example}[theorem]{Example}

\theoremstyle{remark}
\newtheorem{remark}{Remark}

\def\cEE#1{\textcolor{black}{#1}}
\def\IF{\infty}
\def\KK{E}
\def\LT{\left}
\def\RT{\right}
\def\rw{\rightarrow}

\def\ksn#1{\textcolor{black}{#1}}

\usepackage[shortlabels]{enumitem}

\usepackage[nameinlink, noabbrev]{cleveref}

\newcommand{\abs}[1]{\left\lvert #1 \right\rvert}

\newcommand{\ABs}[1]{ \biggl \lvert #1 \biggr \rvert}

\DeclarePairedDelimiterXPP\pk[1]{\mathbb{P}}\{ \}{}{ #1}
\DeclarePairedDelimiterXPP\E[1]{\mathbb{E}}\{ \}{}{	#1}
\DeclarePairedDelimiterXPP\ETT[1]{\widetilde{\mathbb{E}}}\{ \}{}{	#1}

\DeclarePairedDelimiterXPP\ind[1]{\mathbb{I}}( ){}{	#1}

\NewDocumentCommand{\ceil}{s O{} m}{%
	\IfBooleanTF{#1} 
	{\left\lceil#3\right\rceil} 
	{#2\lceil#3#2\rceil} 
}
\NewDocumentCommand{\floor}{s O{} m}{%
	\IfBooleanTF{#1} 
	{\left\lfloor#3\right\rfloor}
	{#2\lfloor#3#2\rfloor}
}

\newcommand{\norm}[1]{\lVert  #1 \rVert }

\definecolor{c20}{rgb}{0.,0.7,0.}
\definecolor{c30}{rgb}{0.,0.,1.}
\definecolor{c40}{rgb}{1,0.1,0.7}
\definecolor{c50}{rgb}{1,0,0}
\definecolor{c60}{rgb}{1,0.9,0.1}
\definecolor{c70}{rgb}{0.50,1.00,0.00}

\numberwithin{equation}{section}
\newtheorem{theo}{Theorem}[section]
\newtheorem*{theo*}{Theorem}

\newtheorem{sat}[theo]{Proposition}
\newtheorem{de}[theo]{Definition}
\newtheorem{lem}[theo]{Lemma}

\newtheorem{korr}[theo]{Corollary}

\numberwithin{equation}{section}

\newcommand{\COM}[1]{} 
\def\IF{\infty}

\newcommand{\R}{\mathbb{R}}

\newcommand{\limit}[1]{\lim_{#1 \to   \infty}}

\newcommand{\vk}[1]{\boldsymbol{#1}}

\def\k1#1{{\textcolor{black}{#1}}}

\def\bqn#1{ \begin{eqnarray} #1 \end{eqnarray}}

\newcommand{\BS}{\begin{sat}}
	\newcommand{\ES}{\end{sat}}
\newcommand{\BT}{\begin{theo}}
	\newcommand{\ET}{\end{theo}}
\newcommand{\BK}{\begin{korr}}
	\newcommand{\EK}{\end{korr}}

\newcommand{\BEX}{\begin{example}}
	\newcommand{\EEX}{\end{example}}

\newcommand{\BD}{\begin{de}}
	\newcommand{\ED}{\end{de}}

\newcommand{\BIT}{\begin{itemize}}
	\newcommand{\EIT}{\end{itemize}}
\newcommand{\BDI}{\begin{description}}
	\newcommand{\EDI}{\end{description}}

\newcommand{\BRM}{\begin{remark}}
	\newcommand{\ERM}{\end{remark}}

\newcommand{\BEL}{\begin{lem}}
	\newcommand{\EEL}{\end{lem}}

\def\inn{\in \mathbb{N}}

\begin{document}

\title{Parisian ruin of locally self-similar Gaussian processes}

\author{
	\name{
		Svyatoslav ~M. Novikov\thanks{
			CONTACT Svyatoslav ~M. Novikov. Email:
			Svyatoslav.Novikov@unil.ch
		}
	}
	\affil{
		Department of Actuarial Science, University of Lausanne, Unil-Dorigny,
		1015 Lausanne, Switzerland}
}

\maketitle

\begin{abstract}
	We derive exact tail asymptotics of the Parisian ruin probability for Gaussian risk models driven by locally self-similar Gaussian processes with a power-type deterministic trend. The considered setting includes non-stationary Gaussian processes whose local correlation structure is governed by a self-similar limiting process, extending classical fractional Brownian motion models.
	
	The asymptotic behaviour is shown to depend on the interplay between the local variance decay, the self-similarity index, and the trend exponent, leading to several distinct regimes. In each regime, the ruin probability admits an explicit asymptotic representation involving Parisian Pickands-type constants.
	
	The analysis relies on a uniform Pickands lemma allowing for families of limiting Gaussian fields, extending existing double-sum techniques and enabling the treatment of locally self-similar Gaussian risk models.
\end{abstract}


\begin{amscode}
	Primary 60G15;
	Secondary 60G70
\end{amscode}

\section{Introduction}
A classical benchmark stochastic process for modelling the insurance risk process is Brownian motion. In  recent contributions general risk processes are modelled by relying on a general centered Gaussian process 
$X(t)$, $t\in[0,T]$   with almost surely
continuous trajectories and variance function $\sigma^2(t)$. The classical finite-time ruin problem concerns the calculation of the tail of 
\[
M(T)=\sup_{t\in[0,T]} X(t).
\]
Since this cannot be done in closed form for general $X$, the main focus of the literature has been on deriving asymptotic approximations for $\pk{M(T)>u}$ as $u\to\infty$, which is a key topic in extremes of Gaussian processes. The theory of extremes
for non-stationary Gaussian processes has been extensively studied, starting
with the seminal work of Piterbarg and co-authors; see, e.g.,
\citet*{Pit96,high2} and the references therein. A key assumption in this
literature is that the variance function $\sigma^2(t)$ attains its maximum at a
unique point, which we take without loss of generality to be $t_0=0$, and that
the local behaviour of $\sigma(t)$ around this point is known. \\ 
A typical
assumption is
\begin{equation}
\label{sigmaE}
\sigma(0)-\sigma(t) \sim b t^{\beta}, \quad t \downarrow 0,
\end{equation}
for some positive constants $b$ and $\beta$, where $\sim$ denotes asymptotic
equivalence as the argument tends to zero (or infinity, depending on the
context).

For a Gaussian process $Y(t)$, $t\ge 0$, define the variogram
\[
V_Y(t,s):=\mathrm{Var}(Y(t)-Y(s)).
\]
In addition to \eqref{sigmaE}, a second fundamental assumption is the
so-called Pickands condition
\begin{equation}
\label{PickE}
1-r_X(t,s) \sim a V_Y(t,s), \quad s,t \downarrow 0,
\end{equation}
where $r_X(t,s)$ denotes the correlation function of $X$, $a>0$ is a constant,
and $Y$ is a centered Gaussian process.

A classical example arises when $Y$ is a fractional Brownian motion (fBm) with
Hurst index $H=\kappa/2$, $\kappa\in(0,2]$, that is,
$\E{Y^2(t)}=t^{\kappa}$, $t\ge0$, and
\[
V_Y(t,s)=|t-s|^{\kappa}, \quad s,t\ge0.
\]
An important feature of fBm is its self-similarity: for any $a>0$,
\[
\{Y(at):t\ge0\} \stackrel{\mathcal{D}}{=}
\{a^{H}Y(t):t\ge0\}.
\]

In \citet*{tabis}, condition \eqref{PickE} was extended by allowing $Y$ to be a
general self-similar centered Gaussian process. Following that work, a Gaussian
process $X$ satisfying \eqref{PickE} for some self-similar $Y$ is referred to as
\emph{locally self-similar at~$0$}. The tail asymptotics of $M(T)$ were derived
there for a broad class of such processes.

The study of $M(T)$ corresponds to the classical ruin probability in Gaussian
risk models. A natural and practically relevant extension is the Parisian ruin
probability, which has been investigated for Gaussian risk models in
\citet*{MR3457055,dkebicki2015parisianJAP}; see also
\citet*{ji2020cumulative,kriukov2022parisian,krystecki2022parisian,
jasnovidov2020approximation,jasnovidov2024parisianruininsurerreinsurer,PavlePar}.
Specifically, it is defined as
\[
p_{X-c}(\Gamma_{T,L},u)
=\pk*{\Gamma_{T,L}(X-c)>u},
\]
where
\[
\Gamma_{T,L}(f)
=\sup_{t\in[0,T]}\;\inf_{s\in[0,L]} f(t+s), \quad L\ge0,
\]
is the Parisian functional and $c(t)$ is a continuous deterministic trend
function. In insurance applications one typically considers power-type trends of the form
$c(t)=dt^{\gamma}$ with $d,\gamma>0$. Note that $\Gamma_{T,L}$ reduces to the
supremum functional when $L=0$.

The main objective of this paper is to derive exact asymptotics for the Parisian
ruin probability associated with general locally self-similar Gaussian
processes. At first glance, such an extension may appear straightforward in
view of the extensive literature on extremes of Gaussian processes. However,
similarly to the phenomenon observed in \citet*{novikov2025sojourn}, the Parisian
setting turns out to require new methodological developments.

As in the classical case, a key step is the approximation of ruin probabilities
over shrinking intervals (often referred to as Pickands intervals), see
\citet*{PickandsA}. More precisely, one needs sharp asymptotics for
\[
p_{X-c}(\Gamma_{T_u,L_u},u), \quad u\to\infty,
\]
where $T_u$ and $L_u$ tend to zero at rates dictated by the local correlation
structure. Moreover, following the uniform double-sum methodology of
\citet*{Uniform2016}, such approximations must hold uniformly.

In contrast to the classical setting of \citet*{Uniform2016}, where a single
limiting Gaussian field appears, the locally self-similar framework leads to a
family of limiting fields. Consequently, existing uniform Pickands-type results
are not directly applicable. Our main technical contribution is therefore an
extension of \citet*[Theorem~2.2]{Uniform2016} to the case of multiple limiting
Gaussian fields. Informally, this result may be viewed as a \emph{uniform
Pickands lemma with non-unique limits}.

Such a result is of independent interest and can be useful in other contexts,
for example in the study of ruin probabilities for locally stationary Gaussian
processes $X(t)$, $t\in[0,1]$, satisfying
\begin{equation}
\label{localstat}
Cov(X(t),X(s))
=1-c_t(1+o(1))|t-s|^{\alpha}, \quad |t-s|\to0,
\end{equation}
where $c_t>0$ is a continuous function. While the asymptotics of
\[
\pk{\sup_{t\in[0,1]}X(t)>u}
\]
can be obtained via partitioning and Slepian-type arguments
\citep[Theorem~7.1]{Pit96}, such techniques are unavailable for more complex
functionals, including the Parisian functional.

Similar difficulties arise for vector-valued Gaussian processes, where neither
the Slepian inequality nor existing vector versions of the uniform Pickands
lemma apply. Allowing for families of limiting Gaussian fields resolves this
issue and substantially broadens the scope of the method.

The paper is organised as follows. In Section~\ref{the-weak-conv} we establish a
novel Pickands-type lemma for families of limiting Gaussian fields. We then show
that the Parisian Pickands constant defined via fractional Brownian motion can
equivalently be expressed using a general self-similar Gaussian process.
Subsequently, we derive the exact asymptotics of the Parisian ruin probability
for locally self-similar Gaussian processes and illustrate the results with
examples. Auxiliary results are collected in Section~4, while all proofs are
deferred to Section~5.

\section{Main Results}
Given a Gaussian process $Y(t), t \geq 0$ and a function $h(t), t \geq 0$, as in \eqref{Pickands_const} below, we denote
\begin{align*}
&
\mathcal{H}^{\mathrm{par},h}_{Y,T,L} \coloneq 
\E*{\sup_{t \in [0,T]}
	\inf_{s \in [0,L]}
	e^{\sqrt{2}Y(t+s)-\mathrm{Var}(Y(t+s))-h(t+s)}},\,
\\
&
\mathcal{H}^{\mathrm{par}}_{Y,L} :=
\lim_{T \to \infty} \frac{\mathcal{H}^{\mathrm{par},0}_{Y,T,L}}{T}. 
\end{align*}

\cEE{For given $\kappa \in (0,2]$ denote by $B_{\kappa}(t)$ the centered continuous one-dimensional fBm with Hurst parameter $\kappa/2$, i.e., $B_{\kappa}$  is a Gaussian process} satisfying  $B_{\kappa}(0)=0$ and $$\mathrm{Var}(B_{\kappa}(t_1)-B_{\kappa}(t_2))=|t_1-t_2|^{\kappa}, \cEE{t_i\ge 0,i=1,2}.$$	 
If $h=0$, then we suppress the superscript $h$ and write simply $\mathcal{H}^{\mathrm{par}}_{Y,T,L}$
instead of $\mathcal{H}^{\mathrm{par},0}_{Y,T,L}.$

\k1{If $h=0$ and $Y=B_{\kappa}$, then we suppress the subscript $Y$ as well and write simply $\mathcal{H}^{\mathrm{par}}_{\kappa,T,L}$ instead of $\mathcal{H}^{\mathrm{par},0}_{B_{\kappa},T,L}$ and $\mathcal{H}^{\mathrm{par}}_{\kappa,L}$ instead of $\mathcal{H}^{\mathrm{par}}_{B_{\kappa},L}$.} 

\ksn{In the definition of $\mathcal{H}^{\mathrm{par},h}_{Y,T,L}$ we also allow $T=\infty$, denoting the corresponding constant by $\mathcal{H}^{\mathrm{par},h}_{Y,\infty,L}$}.

Now we will emphasize which kind of $Y(t)$ we will consider in \eqref{PickE}. Let $(Y(t))_{t \geq 0}$ be a centered Gaussian process with a.s. continuous sample paths. Set $V_Y(t) \coloneq V_Y(1,t)$. If the following assumptions 
\begin{enumerate}[S1]
	\item \label{LS1}
	$Y(\cdot)$ is self-similar with index $\alpha / 2>0$ and $\mathrm{Var}(Y(1))=1$;
	
	\item \label{LS2} there exist $\kappa \in(0,2]$ and $c_Y>0$ such that $$V_Y(1-h)\sim c_Y|h|^\kappa, \quad h \to 0$$
\end{enumerate}
are satisfied, then we simply write $Y \in \mathbf{S}\left(\alpha, \kappa, c_Y\right)$. 

As in \citet*{MR4206416}, we introduce the class of locally self-similar Gaussian processes, which are characterized by a local property at $t=0$.
\BD A centered Gaussian process $X(t), t \geq 0$ is said to be locally self-similar at $0$ with parameters $a$ and $Y$, if \eqref{PickE} holds for some $a>0, Y \in \mathbf{S}\left(\alpha, \kappa, c_Y\right)$. 
\ED

\begin{remark}
	If the process $Y(t)$ was defined for negative $t$, we could have defined self-similarity at a point $t_0 \in \R$ as well (this way, the point being exactly $0$ makes no difference). Indeed, in \eqref{PickE} it is enough to put $s,t \to t_0$ instead of $s,t \downarrow 0$.
\end{remark} 

In \citet*[Theorem 3.4]{MR4206416} a relation between the Pickands constant of a self-similar Gaussian process $Y$ and the classical Pickands constant was proved. We extend the aforementioned result to the Parisian functionals. In our setting there is an additional technical difficulty due to the lack of the Slepian inequality for our functional.\\ 

Let 
\begin{align}
\label{betahat}
\hat{\beta}\coloneq\frac{\beta \kappa}{\alpha},\;\hat{\gamma}\coloneq\frac{\gamma \kappa}{\alpha}
\end{align}
for 
given $\beta,\gamma>0$. Further, given a Gaussian process $Z(t),t\ge 0$ denote 
\begin{align}
\label{Zhat}
\widehat{Z}(t)\coloneq Z(t^{\kappa/\alpha}).
\end{align}  It follows that  $\widehat{Y} \in \mathbf{S}(\kappa,\kappa,c_{\widehat{Y}})$ with $c_{\widehat{Y}}=c_Y(\kappa/\alpha)^{\kappa}$. Indeed,
\begin{align*}
	V_{\widehat{Y}}(1-h)
	&=V_Y((1-h)^{\kappa/\alpha})
	=V_Y(1-(\kappa/\alpha)h+o(h))
	\\
	&=c_Y |(\kappa/\alpha)h+o(h)|^\kappa
	\sim c_Y (\kappa/\alpha)^\kappa |h|^\kappa, \quad h \to 0.
\end{align*}

Recall that in our notation
\begin{align*}
	&
	\mathcal{H}^{\mathrm{par}}_{\widehat{Y},L}
	=\lim_{T \to \infty}\frac{\E*{\sup_{t \in [0,T]}
		\inf_{s \in [0,L]}
		e^{\sqrt{2}\widehat{Y}(t+s)-\mathrm{Var}(\widehat{Y}(t+s))}}}{T},
	\\
	& 
	\mathcal{H}^{\mathrm{par}}_{\kappa,L}
	=\lim_{T \to \infty}\frac{\E*{\sup_{t \in [0,T]}
			\inf_{s \in [0,L]}
			e^{\sqrt{2}B_\kappa(t+s)-\mathrm{Var}(B_\kappa(t+s))}}}{T}.
\end{align*}

\BT \label{last}
If  $\widehat{Y} \in \mathbf{S}(\kappa,\kappa,c_{\widehat{Y}})$ and $L \geq 0$, \cEE{then}   \bqn{ \mathcal{H}^{\mathrm{par}}_{\widehat{Y},L}=c_{\widehat{Y}}^{1/\kappa}   \mathcal{H}^{\mathrm{par}}_{\kappa,c_{\widehat{Y}}^{1/\kappa}L} \in (0,\infty).
}
\ET

We extend \citet*[Theorem 4.2]{tabis} to the case of Parisian functional.
\k1{ Hereafter $\mathbb{I}(\cdot)$ stands for the indicator function.}

\BT\label{mainth}
If $X(t), t \geq 0$ is locally self-similar at $0$ with parameters $a$ and 
\\$Y \in \textbf{S}(\alpha,\kappa,c_Y)$ \ksn{which satisfies \eqref{sigmaE}}, then we have 
\begin{align} 
\label{mainasymp}
\pk*{\underset{t \in [0,T]}{\sup}\; \underset{s \in [0,L_u]}{\inf} \cEE{\Bigl(} X(t+s)-d(t+s)^{\gamma} 
	\cEE{\Bigr)} > u } \sim c u^{\cEE{p}}\Psi(u), \quad u\to \IF
\end{align} under the following conditions 
\begin{enumerate}[(i)]
	\item  If $\alpha < \min(\beta,2\gamma),\;\alpha \leq \kappa$, then for
	$L_u=Lu^{-2/\alpha-\frac{\alpha-\kappa}{\alpha}(2/\kappa-\max(2/\hat{\beta},1/\hat{\gamma}))}$, $L\geq 0$ 
	\eqref{mainasymp} holds with
	\\ $p=2/\kappa-\max(2/\hat{\beta},1/\hat{\gamma})$,
	\begin{align*}
	c=(ac_Y)^{1/\kappa}		
	\int\limits_{0}^{\infty} 
	\exp\cEE{\Bigl(}-d\mathbb{I}(2\gamma \leq \beta) z^{\gamma}
	-b\mathbb{I}(\beta\leq 2\gamma) z^{\beta}\cEE{\Bigr)}z^{\alpha/\kappa-1} \mathcal{H}^{\mathrm{par}}_{\kappa,(ac_Y)^{1/\kappa}\frac{\alpha}{\kappa} L  z^{\alpha/\kappa-1}} dz.
	\end{align*}
	
	\item  If $\alpha<\min(\beta,2\gamma)$ and $\alpha>\kappa$, then for $L_u=Lu^{-2/\alpha-\epsilon}$, 
	\\$\epsilon \in (0,\frac{\alpha-\kappa}{\alpha}(2/\kappa-\max(2/\hat{\beta},1/\hat{\gamma}))]$, $L>0$ \eqref{mainasymp} holds with $p=\frac{\epsilon\alpha}{\alpha-\kappa}$,
	\begin{align*}
	c&=(ac_Y)^{1/\kappa}		
	\int\limits_{0}^{\infty} 
	\exp\cEE{\Bigl(}-d\mathbb{I}(\epsilon =\frac{\alpha-\kappa}{\alpha}(2/\kappa-1/\hat{\gamma})) z^{\gamma}
	-b\mathbb{I}(\epsilon =\frac{\alpha-\kappa}{\alpha}(2/\kappa-2/\hat{\beta})) z^{\beta}\cEE{\Bigr)}
	\\
	&\quad \times
	z^{\alpha/\kappa-1} \mathcal{H}^{\mathrm{par}}_{\kappa,(ac_{Y})^{1/\kappa}
		\frac{\alpha}{\kappa} L  z^{\alpha/\kappa-1}} dz.
	\end{align*}
	If $\epsilon =\frac{\alpha-\kappa}{\alpha}(2/\kappa-\max(2/\hat{\beta},1/\hat{\gamma}))$, then the same works with $L=0$.
	
	\item  If $\alpha \geq \min(\beta,2\gamma)$ or $\alpha>\kappa$, then for $L_u=Lu^{\min\left(-2/\alpha,-2/\beta,-1/\gamma\right)}$, $L>0$ \eqref{mainasymp} holds with $p=0$, \\$$c=\mathcal{H}^{\mathrm{par},h}_{Y\mathbb{I}(\alpha\leq\min\left(\beta,2\gamma\right)),\ksn{\infty},a^{1/\alpha}L},$$ with 
	\[
	h(t)=a^{-\beta/\alpha}bt^{\beta}\mathbb{I}(\beta\leq\min\left(\alpha,2\gamma\right))+a^{-\gamma/\alpha}dt^{\gamma}\mathbb{I}\left(2\gamma\leq \min\left(\alpha,\beta\right)\right).
	\] 
	
	If $\alpha \geq \min(\beta,2\gamma)$, then the same works for $L=0$. In particular, the corresponding constants are well-defined, positive and finite.
\end{enumerate}	
\begin{remark}
	\label{bigalpha}
	If $\alpha>\min(\beta,2\gamma)$, then from the proof of case (iii) one can see that for \Cref{mainth} to hold  it is enough to have
	\[
	\limsup_{t,s \downarrow 0} \frac{1-r_X(t,s)}{V_Y(t,s)}<\infty
	\]
	instead of local self-similarity with parameters $a$ and $Y$, and \eqref{mainasymp} holds with
	\begin{align}
	\label{degconst}
	c=e^{-\mathbb{I}(\beta \leq 2\gamma)b L^{\beta}-\mathbb{I}(2\gamma \leq \beta)d L^{\gamma}}.
	\end{align} 
	This constant coincides with the constant $c$ given by case (iii) of \Cref{mainasymp}, because $Y \mathbb{I}(\alpha \leq \min(\beta,2 \gamma))=0$, so	
	\begin{align*}
	\mathcal{H}^{\mathrm{par},h}_{Y \mathbb{I}(\alpha \leq \min(\beta,2 \gamma)),\ksn{\infty},a^{1/\alpha}L}&=
	\E{\sup_{t \in [0,\infty]} \inf_{s \in [0,a^{1/\alpha}L]} e^{-h(t+s)}}
	\\
	&
	=e^{-h(a^{1/\alpha}L)}=e^{-\mathbb{I}(\beta \leq 2\gamma)b L^{\beta}-\mathbb{I}(2\gamma \leq \beta)d L^{\gamma}}.
	\end{align*} 
\end{remark}
\begin{remark}
	Scenarios (ii) and (iii) do not overlap. Indeed, the only possible situation of overlap could be the case $\alpha>\kappa,\,\alpha < \min(\beta,\kappa)$. But then $L_u=Lu^{-2/\alpha-\epsilon}<Lu^{-2/\alpha}$ in scenario (ii) and $L_u=Lu^{-2/\alpha}$ in scenario (iii).
\end{remark}
\begin{remark}
	In cases (i) and (ii) we have $L_u=L u^{-\frac{2+p(\alpha-\kappa)}{\alpha}}$. The intuitive explanation of this is as follows: in fact during the proof we will consider the Parisian ruin of $Z(t)=X(t^{\kappa/\alpha})-d t^{\hat{\gamma}}$. The main contribution to the Parisian ruin probability will come from $t$ of order
	$u^{-2/\kappa+p}$, and we will split the set of such $t$ into intervals of length of order $u^{-2/\kappa}$ (which could be referred to as Pickands intervals) and apply \Cref{the-weak-conv-cor} (with a suitable time rescaling) for each interval. In order to do this we should pick $L_u$ so that the intervals of length $L_u$ for $X(t)-dt^{\gamma}$ correspond to intervals of length of the same order $u^{-2/\kappa}$ for $Z(t)$ with $t$ of order $u^{-2/\kappa+p}$. 
	
	Let $[t,t+\tau]$ be one of such intervals for $Z(t)$, then (since $Z(t)=X(t^{\kappa/\alpha})-dt^{\hat{\gamma}}$) it corresponds to $[t^{\kappa/\alpha},(t+\tau)^{\kappa/\alpha}]$ for $X(t)$. The length of the latter interval is
	\[
	(t+\tau)^{\kappa/\alpha}-t^{\kappa/\alpha}
	=t^{\kappa/\alpha}
	((1+\tau/t)^{\kappa/\alpha}-1)
	\approx
	\frac{\kappa}{\alpha} t^{\kappa/\alpha-1} \tau.
	\]    
	Since $t$ is of order $u^{-2/\kappa+p}$ and $\tau$ is of order $u^{-2/\kappa}$, $L_u$ should be of order $u^{(-2/\kappa+p)(\kappa/\alpha-1)-2/\kappa}=u^{-\frac{2+p(\alpha-\kappa)}{\alpha}}$ as mentioned above.
\end{remark}
\begin{remark}
	1) For cases (i) and (ii) in the constant $\mathcal{H}^{\mathrm{par}}$ the value $ \frac{\alpha}{\kappa}L$ appears, while in \citet*[Theorem 2.2]{novikov2025sojourn} only $L$ appears due to some typos.
	
	2) For case (iii) in the constant $\mathcal{H}^{\mathrm{par}}$ the value $a^{1/\alpha}L$ appears, while in \citet*[Theorem 2.2]{novikov2025sojourn} only $L a^{1/\kappa}$ appears due to some typos.
\end{remark}
The most natural way to construct locally self-similar processes is to start from a Gaussian process $\widetilde{Y} \in \mathbf{S}(\widetilde{\alpha},\widetilde{\kappa},c_{\widetilde{Y}})$. This approach gives two corollaries from \Cref{mainth}.	
\begin{corollary}\label{corselfsim_2}
If $\widetilde{Y} \in \mathbf{S}\left(\alpha, \kappa, c_{\widetilde{Y}}\right)$, then
the asymptotics from \Cref{mainth} hold for 
\[
X(t)=
\begin{cases}
	\widetilde{Y}(1-t), t \in [0,1]\\
	0,t>1
\end{cases}
\] 
with $T=1$, $Y(t)=B_\kappa(t)$, $\alpha=\widetilde{\kappa}$, $\kappa=\widetilde{\kappa}$, $a=\begin{cases}
	c_{\widetilde{Y}}/2 \text{ for } \kappa<2,\\
	\frac{c_{\widetilde{Y}}-\alpha^2/4}{2 } \text{ for } \kappa=2 
\end{cases}$, $b=\alpha/2$, $\beta=1$. 
\end{corollary}
\begin{corollary}\label{corselfsim}
	If $\widetilde{Y} \in \mathbf{S}\left(\widetilde{\alpha}, \widetilde{\kappa}, c_{\widetilde{Y}}\right)$, $V_{\widetilde{Y}}(x)$ is decreasing in some neighborhood of $0$, $V_{\widetilde{Y}}(x)$ attains its
	maximum on $[0,1]$ at the unique point $x=0$ and $1-V_{\widetilde{Y}}(x)\sim Rx^{\widetilde{\beta}}$ with $\widetilde{\beta} \geq 1$, and, in addition, the function $\frac{\frac{\partial}{\partial x}V_{\widetilde{Y}}(x) }{x^{\widetilde{\beta}-1}}$ is bounded, then
	the asymptotics from \Cref{mainth} hold for $X(t)=\widetilde{Y}(1)-\widetilde{Y}(t)$ with $T=1$,
	$Y(t)=\widetilde{Y}(t)$, $\alpha=\widetilde{\alpha}$,
	$\kappa=\widetilde{\kappa}$, $a=1/2$, $b=R/2$, $\beta=\widetilde{\beta}$.
\end{corollary}
\begin{remark}
	If $\beta\leq\alpha/2$, then we cannot ensure the condition \eqref{PickE}. However, this case falls under \Cref{bigalpha}. Therefore, since, as $t,s \to 1$,
	\[
		\frac{1-r_X(t,s)}{V_{\widetilde{Y}}(t,s)}
		=\frac{1}{2} \cdot \frac{V_X(t,s)-(\sigma_X(t)-\sigma_X(s))^2}{V_{\widetilde{Y}}(t,s)\sigma_X(t)\sigma_X(s)}  
		\leq 
		\frac{V_X(t,s)}{2 V_{\widetilde{Y}}(t,s)\sigma_X(t)\sigma_X(s)}=\frac{1}{2}+o(1),
	\]
	\Cref{mainth} is applicable, allowing us to deduce Corollary \ref{corselfsim}.

	If, in turn, $\beta>\alpha/2$, then one can check that
	\[
		\frac{(\sigma_X(t)-\sigma_X(s))^2}{V_{\widetilde{Y}}(t,s)\sigma_X(t)\sigma_X(s)} \to 0,\, t,s \to 0,
	\]
	hence, \eqref{PickE} holds with $a=1/2$ and again \Cref{mainth} is applicable.
\end{remark}
\ET

\section{Examples}
In this section similarly to \citet*{novikov2025sojourn}
we consider different possible examples of $Y$ in Corollary \ref{corselfsim}. 
\k1{ Below $R_Y$ stands for the covariance function of the Gaussian process $Y$.}
\cEE{The set of examples considered here appear in \citet*{tabis}, \citet*{MR4206416} and \citet*{novikov2025sojourn}, therefore we omit some details.}
\BEX
For $\alpha \in (1,2)$ consider the covariance function
$$
R_Y(t,s)\coloneq\frac{(t+s)^{\alpha}-|t-s|^{\alpha}}{2^{\alpha}}.
$$

In view of \citet*{tabis} $\beta=1,R=\alpha 2^{2-\alpha}, Y \in \mathbf{S}(\alpha,\alpha,2^{1-\alpha})	$. Let $\gamma=\beta/2=1/2$, then \eqref{mainasymp} holds with $L \geq 0$, $L_u=Lu^{-2}$, $p=0$ and \ksn{by \eqref{degconst}} $c=e^{-\alpha 2^{1-\alpha}L-dL^{1/2}}$.
\EEX

\BEX
Sub-fractional Brownian motion with parameter $\alpha \in \ksn{(1,2)}$ is a centered Gaussian process with covariance function
$$
R_Y(t,s)\coloneq\frac{1}{2-2^{\alpha-1}}
\left(t^{\alpha}+s^{\alpha}-
\frac{(t+s)^{\alpha}+|t-s|^{\alpha}}
{2}\right).
$$
In view of \citet*{tabis} $\beta=\alpha,R=\frac{2^{\alpha-1}}{2-2^{\alpha-1}}, Y \in \mathbf{S}(\alpha,\alpha,(2-2^{\alpha-1})^{-1})	$. Let $\gamma=\beta/2$, then \eqref{mainasymp} holds with $L \geq 0$, $L_u=Lu^{-2/\alpha}$, $p=0$ and $c=\mathcal{H}^{\mathrm{par},h}_{Y,\ksn{\infty},2^{-1/\alpha}L}$, where \\$h(t)=\frac{2^{\alpha-1}}{2-2^{\alpha-1}}t^{\alpha}+\sqrt{2}dt^{\alpha/2}$.
\EEX
\begin{remark}
	The condition $\alpha \geq 1$, which corresponds to the requirement $\beta \geq 1$ in \Cref{corselfsim}, was mistakenly not included in Example 3.2 of \citet*{novikov2025sojourn}.
\end{remark}

\BEX
Negative sub-fractional Brownian motion $Y$ with parameter \ksn{$\alpha \in(2,4]$} is a centered Gaussian process with covariance function
$$
R_Y(t, s)\coloneq\frac{1}{2^{\alpha-1}-2}\left(\frac{(t+s)^\alpha+|t-s|^\alpha}{2}-t^\alpha-s^\alpha\right) .
$$
In view of \citet*{tabis}	
$\beta=2,R=\frac{\alpha(\alpha-1)}{2^{\alpha-1}-2},Y \in \mathbf{S}(\alpha,2,\frac{\alpha(\alpha-1)2^{\alpha-3}}{2^{\alpha-1}-2} )$. Let $\gamma=\beta/2=1$, then
\eqref{mainasymp} holds with $L_u=Lu^{-1}$, $p=0$ and \ksn{by \eqref{degconst}} \[c=e^{-\frac{\alpha(\alpha-1)}{2^{\alpha}-4}L^2-dL}.\]
\EEX
\COM{
	\begin{remark}
		The condition $\alpha < 4$, which corresponds to the requirement $\beta > \alpha/2$ in \Cref{corselfsim}, was not included in Example 3.3 of \citet*{novikov2025sojourn}. However, the case $\beta=\alpha/2$ corresponds to the case $(iii)$ of \Cref{mainth}, and in this case it is actually not required that $1-r_X(t,s) \sim a V_Y(t,s)$, where $r_X$ is the correlation function of $X$. It is enough to get an upper bound on $1-r_X(t,s)$.
	\end{remark}
}
\BEX Weighted fBm $Y$ with parameters $\kappa \in(0,2]$ and $a>1$ is a centered Gaussian process with covariance function
$$
R_Y(t, s)\coloneq\frac{\Gamma(a+\kappa)}{2 \Gamma(a) \Gamma(\kappa)} \int_0^{\min (t, s)} u^{a-1}\left((t-u)^{\kappa-1}+(s-u)^{\kappa-1}\right) d u .
$$
In view of \citet*{tabis}
$\beta = a,R=\frac{\Gamma(a+\kappa)}{\Gamma(a+1)\Gamma(\kappa)}, Y \in \mathbf{S}(a+\kappa-1,\kappa,\frac{\Gamma(a+\kappa)}{\Gamma(a)\Gamma(\kappa+1)}) 	$. Let $\gamma=\frac{\beta}{2}=\frac{a}{2}$, denote $\alpha\coloneq a+\kappa-1.$

If $\kappa>1$, then $\alpha>\beta$ and
\eqref{mainasymp} holds with $L_u=Lu^{-2/a}$, $p=0$ and \ksn{by \eqref{degconst}} \[c=e^{-\frac{\Gamma(a+\kappa)}{2\Gamma(a+1)\Gamma(\kappa)}L^a-dL^{a/2}}.\]

If $\kappa =1$, then $Y(t)$ has the same distribution as $B_1(t^a)$, $\alpha=\beta$ and \eqref{mainasymp} holds with $L \geq 0$, $L_u=Lu^{-2/a}$, $p=0$ and $c=\mathcal{H}_{Y,\infty,2^{-1/a}L}^{\mathrm{par},h}$, where $h(t)=t^a+\sqrt{2}dt^{a/2}$.

If $\kappa<1$, then $\alpha<\beta,\,\alpha>\kappa$, so \eqref{mainasymp} holds with \[L_u=Lu^{-\frac{2}{a+\kappa-1}-\epsilon},\epsilon \in \Big(0,\frac
{2(a-1)(1-\kappa)}
{ (a+\kappa-1)a \kappa}\Big],p=\epsilon \frac{a+\kappa-1}{a-1}
\] and

\begin{align*}
c&=&\left(\frac{\Gamma(a+\kappa)}{2\Gamma(a)\Gamma(\kappa+1)}\right)^{1/\kappa}
\int\limits_{0}^{\infty} 
\exp\left(-d z^{a/2}
-\frac{\Gamma(a+\kappa)}{2\Gamma(a+1)\Gamma(\kappa)} z^{a}\right)z^{\frac{a-1}{\kappa}}
\\
& \times \mathcal{H}^{\mathrm{par}}_{\left(\frac{\Gamma(a+\kappa)}{2\Gamma(a)\Gamma(\kappa+1)}\right)^{1/\kappa}\frac{a+\kappa-1}{\kappa} L  z^{\frac{a-1}{\kappa}}} dz,
\end{align*}
with $L \geq 0$ if $\epsilon = \frac
{2(a-1)(1-\kappa)}
{ \kappa a(a+\kappa-1)}$, and

\begin{align*}
c=\left(\frac{\Gamma(a+\kappa)}{2\Gamma(a)\Gamma(\kappa+1)}\right)^{1/\kappa}
\int\limits_{0}^{\infty} 
z^{\frac{a-1}{\kappa}}
\mathcal{H}^{\mathrm{par}}_{\left(\frac{\Gamma(a+\kappa)}{2\Gamma(a)\Gamma(\kappa+1)}\right)^{1/\kappa}\frac{a+\kappa-1}{\kappa} L  z^{\frac{a-1}{\kappa}}} dz,
\end{align*}
with $L>0$ if $0 <\epsilon <\frac
{2(a-1)(1-\kappa)}
{ \kappa a (a+\kappa-1) }$,
\[
c=\mathcal{H}^{\mathrm{par},h}_{Y,\infty,2^{-1/(a+\kappa-1)}L},
\]
where $h(t)=2^{\frac{1-\kappa}{a+\kappa-1}}\frac{\Gamma(a+\kappa)}{\Gamma(a+1)\Gamma(\kappa)}t^a+2^{\frac{a}{2(a+\kappa-1)}}d t^{a/2}$ with $L>0$ if $\epsilon=0$.
\EEX

\BEX
Integrated fBm $Y$ with parameter $\alpha \in(0,2]$ is a Gaussian process defined as
\[
Y(t)=\sqrt{\alpha+2}\int_0^t B_\alpha(s) ds.
\]

Its
covariance function is a function of the form
$$
R_Y(t, s)=\frac{(\alpha+2)\left(s^{\alpha+1} t+s t^{\alpha+1}\right)+|t-s|^{\alpha+2}-t^{\alpha+2}-s^{\alpha+2}}{2(\alpha+1)} .
$$
In view of \citet*{tabis} $Y \in \mathbf{S}(\alpha+2,2,\alpha+2)$ and	 $\beta=\alpha+1$ if $\alpha \leq 1$ and $\beta = 2$, if $\alpha \geq 1$. Let $\gamma = \beta/2$, then for some $b>0$ \eqref{mainasymp} holds with $L_u=Lu^{-\max\left(1,\frac{2}{\alpha+1}\right)}$, $p=0$ and \ksn{by \eqref{degconst}} \[c=e^{-bL^{\min\left(\alpha+1,2\right)}-dL^{\min\left((\alpha+1)/2,1\right)}}.\]
\EEX

\BEX
Time-average of fBm $Y$ with parameter $\alpha \in [1,2]$ is a Gaussian process defined as
$$
Y(t)\coloneq\sqrt{\alpha+2} \frac{1}{t} \int_0^t B_\alpha(s) d s .
$$
In view of \citet*{tabis} $\beta=1$, $Y \in \mathbf{S}(\alpha,2,1)$,	 and $R=\alpha/2+1$, if $\alpha>1$, $R=2$, if $\alpha=1$. Let $\gamma=\beta/2=1/2$. 
If $\alpha>1$, then \eqref{mainasymp} holds with $L_u=Lu^{-2}$, $p=0$ and \ksn{by \eqref{degconst}}
$c=e^{-\left(\frac{\alpha+2}{4} \right)L-dL^{1/2}}$. 

If $\alpha =1$, then \eqref{mainasymp} holds with $p=0$ and $c=\mathcal{H}^{\mathrm{par},h}_{Y,\infty,L/2}$, where
$h(t)=2t+\sqrt{2}dt^{1/2}$.
\EEX

\BEX
Dual fBm $Y$ with parameter $\alpha \in [1,2]$ is a centered
Gaussian process with covariance function 
$$
R_Y(t,s)\coloneq\frac{t^{\alpha}s+s^{\alpha}t}{t+s}.
$$
In view of \citet*{tabis} $Y \in \mathbf{S}(\alpha,2,\alpha/2)$ and $\beta=1$ and $R=2$, if $\alpha>1$, $R=3$, if $\alpha=1$. Let $\gamma=\beta/2=1/2$. 

If $\alpha>1$, then \eqref{mainasymp} holds with $L_u=Lu^{-2}$, $p=0$ and \ksn{by \eqref{degconst}} $c=e^{-L-dL^{1/2}}$.

If $\alpha =1$, then \eqref{mainasymp} holds with $L_u=Lu^{-2}$, $p=0$, $c=\mathcal{H}^{\mathrm{par},h}_{Y,\infty,L/2}$, where $h(t)=3t+\sqrt{2}dt^{1/2}$.
\EEX

\def\Esu{E(u,n)}

\section{Auxiliary lemmas}
In the sequel we derive some results needed in the proofs below. 

\subsection{A new Pickands lemma}

We shall start with formulating the \emph{uniform Pickands lemma with multiple limits} mentioned in the Introduction in full generality considering 
a family of centered Gaussian processes $\xi_{u,j}(\vk{t}),  \vk{t}\in \KK, \ j\in S_u, \ {u\geq 0}$
with continuous sample paths and variance function $\sigma_{u,j}^2$ defined over the parameter space  $\KK,$ which is a \ksn{compact} subset of $\R^k, k\inn$ that contains the origin of $\R^k$. 

The motivation for such a general formulation stems from the need of $\vk{t}$ being vector-valued in some ruin problems of actuarial relevance going beyond the scope of this paper (for example, the double crossing problem considered in \cite{Ievlev2024}).

Denote by $C(E)$ \cEE{the Banach space of  all continuous functions $f: E \to \R$} equipped with supremum norm.
Write $\norm{\cdot}$ for a norm on $\R^k$ and set \[\omega_f(\delta)\coloneq\sup_{\norm{\vk{t}-\vk{s}} \leq \delta} \abs{ f(\vk{t})-f(\vk{s})}.\] The following assumptions on the functional  $\Gamma: C(E)\rw \R$ are the same as the ones appearing in \citet*{Uniform2016}, only F3 slightly differs from their F2:
\begin{enumerate}[F1]
	\item \label{f0} $\Gamma$ is continuous; 
	\item  \label{f1} $\Gamma(af+b)=a\Gamma(f)+b$ holds for all $f\in C(E),$ all \cEE{ $a>0$ and all $b\in\mathbb{R}$};
	\item   \label{f2} There exist $c_1,c_2>0$ such that 
	$$\Gamma(f)\leq \max(c_1\sup_{\vk{s}\in E}f(\vk{s}),c_2), \ \ \forall f\in C(E).$$
\end{enumerate}

\cEE{Hereafter $\Psi=1- \Phi$, with $\Phi$ the distribution function of a $N(0,1)$ random variable. }

For a random variable $Z$, we set $\overline{Z}\coloneq\frac{Z}{\sqrt{\mathrm{Var}(Z)}}$ if $\mathrm{Var}(Z)>0$.

We state next our first result.
\BT\label{the-weak-conv}
Let $\xi_{u,j}( \vk{s}), \vk{s}\in E ,j\in S_u, {u\geq 0}$ be a family of centered Gaussian processes defined as above and suppose that $\Gamma$ satisfies \Cref{f0}-\Cref{f2}. \cEE{Assume the validity of the following conditions}
\begin{enumerate}[C1]
	\item 
	Let \label{C0}$g_{u,j}, j\in S_u$ be a sequence of deterministic functions of $u$ satisfying
	\begin{align*}
	\lim_{u\to\IF}\inf_{j\in S_u}g_{u,j}=\IF.
	\end{align*}
	\item \label{C1}
	Suppose that $\mathrm{Var}(\xi_{u,j}(0))=1$ for all large $u, j\in S_u$ and
	there exist bounded continuous functions $h_{u,j}$ on $ \KK$ such that 
	for some $u_0>0$ and $M>0$ 
	$$
	\sup_{u \geq u_0,j \in S_u} |h_{u,j}(0)| \leq M , \quad 	\lim_{\delta \downarrow 0}\sup_{u \geq u_0} \sup_{j \in S_u} \omega_{h_{u,j}}(\delta) =0
	$$
	such that
	\begin{align}\label{assump-cova-field}
	\lim_{u\to\IF}\sup_{\vk{s}\in E ,j\in S_u}\left|g_{u,j}^2\LT( 1-\sigma_{u,j}(\vk{s}) \RT) - h_{u,j}(\vk{s})\right| =0.
	\end{align}	
	\item \label{C2} There exist centered Gaussian processes  $\zeta_{u,j}(\vk{s}),\vk{s}\in \R^k$ with a.s. continuous sample paths and \ksn{$\zeta_{u,j}(\vk{0})=0$} such that
	\begin{align}\label{C21}
		\lim_{u\to\IF}\sup_{s, s'\in E, j\in S_u}\abs{g_{u,j}^2\big(\mathrm{Var}(\overline\xi_{u,j}(\vk{s})-\overline\xi_{u,j}(\vk{s}'))\big) - 2\mathrm{Var}(\zeta_{u,j}(\vk{s})-\zeta_{u,j}(\vk{s}'))} =0.
	\end{align}
	\item  \label{C3} There exist positive constants $C, \nu, u_0$ such that
	\begin{align}\label{assump-holder-field}
	\sup_{j\in S_u} g_{u,j}^2\mathrm{Var}(\overline\xi_{u,j}(\vk{s})-\overline\xi_{u,j}(\vk{s}')) \leq C \norm{\vk{s}-\vk{s}'}^\nu, \quad 
	\sup_{j\in S_u} \mathrm{Var}(\zeta_{u,j}(\vk{s})-\zeta_{u,j}(\vk{s}')) \leq C \norm{\vk{s}-\vk{s}'}^\nu
	\end{align}
	holds for all $\vk{s},\vk{s}'\in \KK , u\geq u_0$.
\end{enumerate}
If further 
$\pk*{  \Gamma(\xi_{u,j})> g_{u,j}} > 0$ for  all large
$u$ and all $j \in S_u$, then 
\begin{align}\label{con-uni-con}
\lim_{u\to\IF}\sup_{j\in S_u} \ABs{ \frac{\pk*{ \Gamma(\xi_{u,j})>g_{u,j} } } {\Psi(g_{u,j})} -\E{ e^{ 			\Gamma(\sqrt{2}\zeta_{u,j}-\mathrm{Var}(\zeta_{u,j})-h_{u,j})}}}=0. 
\end{align} 
Moreover, if 
\begin{align}
\label{bddvar}
\sigma^2:=\sup_{\vk{t} \in E,j \in S_u, u \geq u_0}\mathrm{Var}(\zeta_{u,j}(\vk{t})) < \infty,
\end{align}
then
\begin{align}\label{bddconst}
\sup_{u \geq u_0,j\in S_u} \E*{
	e^{
		\Gamma(\sqrt{2}\zeta_{u,j}-\mathrm{Var}(\zeta_{u,j})-h_{u,j})
	}
}<\infty.
\end{align}
\ET

Define the following constant
\bqn{   
	\label{Pickands_const}
	\mathcal{H}^{\mathrm{par},h}_{\zeta,T,L} \coloneq\E*{\sup_{t \in [0,T]}
		\inf_{s \in [0,L]}
		e^{\sqrt{2}\zeta(t+s)-\mathrm{Var}(\zeta(t+s))-h(t+s)}}.
}

Applied to the case of Parisian functional \Cref{the-weak-conv} implies: 

\begin{corollary}
	\label{the-weak-conv-cor}
	Under the assumptions and the notation of \Cref{the-weak-conv} (with $E=[0,T+L]$)
	\cEE{if for  all} large
	$u$ and all $j \in S_u$
	$\pk*{  \sup_{t \in [0,T]} \inf_{s \in [0,L]}\xi_{u,j}(t+s)> g_{u,j}} > 0$ for some $T$ positive and $L$ non-negative, then
	\begin{align}\label{con-uni-con-cor}
	\lim_{u\to\IF}\sup_{j\in S_u} \left| \frac{\pk*{ \sup_{t \in [0,T]} \inf_{s \in [0,L]}\xi_{u,j}(t+s)>g_{u,j} } } {\Psi(g_{u,j})} -\mathcal{H}^{\mathrm{par}, h_{u,j}}_{\zeta_{u,j},T,L} \right| = 0.
	\end{align}
\end{corollary}

\subsection{Additional lemmas}
We shall assume that 
\cEE{$Y \in     \mathbf{S}(\alpha,\kappa,c_{Y})$.}\\
Recall that (see the definitions \eqref{betahat}, \eqref{Zhat}) 
\[
\hat{\beta}=\frac{\beta \kappa}{\alpha},\;\hat{\gamma}=\frac{\gamma \kappa}{\alpha},\;\widehat{Y}(t)=Y(t^{\kappa/\alpha}).
\]

\COM{
\cEE{Below we consider a shorter interval that depends on $u$. This is necessary 
	since }for $t\gg u^{\min\left(-1/\hat{\gamma},-2/\hat{\beta},-2/\kappa\right)+p}$ in \Cref{the-weak-conv-cor} we are not able to \cEE{obtain} a limiting Gaussian process.
}

Denote 
\begin{align}
\label{lut}
\widehat{L}_{u,t}\coloneq(t^{\kappa/\alpha}+L_u)^{\alpha/\kappa}-t.
\end{align}

Let $X(t),\,t \geq 0$ be locally self-similar at $0$ with parameters $a$ and $Y$ and satisfy \eqref{sigmaE} as in the statement of \Cref{mainth}.

For a random variable $Z$, we set $\overline{Z}=\frac{Z}{\sqrt{\mathrm{Var}\left(Z\right)}}$ if $\mathrm{Var}\left(Z\right)>0$.

\BEL \label{Lem25}
Suppose that $ \liminf_{u \to \infty} f(u)/u, \limsup_{u \to \infty} f(u)/u \in (0,\infty)$. There exist absolute constants $\delta, F, G > 0 $ such that

\begin{align*} 
	\pk*{ 
		\begin{aligned}
			&
			\sup_{t \in [A,A+T]u^{-2/\kappa}} \overline{X}(t^{\kappa/\alpha}) > f(u) ,
			\\&
			\sup_{t \in [t_0,t_0+T]u^{-2/\kappa}} \overline{X}(t^{\kappa/\alpha}) > f(u) 
		\end{aligned}
	} 
	\leq \; F T^2 \exp \left( - G (t_0-(A+T))^\kappa \right) \Psi(f(u)) 
\end{align*}
for all $T \geq 1$, $ t_0 > A + T > A > 0 $ and any $ u \geq u_0 = (4\delta)^{-\kappa/2} (t_0 + T)^{\kappa/2} $, i.e. $ (t_0 + T)u^{-2/\kappa} \leq 4\delta $.
\\
\EEL

\BEL \label{Lem27} 
There exist absolute constants $\delta, F, G > 0 $ such that 
\begin{align*}
	\pk*{
		\begin{aligned}
			&
			\sup_{t \in [A,A+T]u^{-2/\kappa}} \overline{X}(t^{\kappa/\alpha}) > u , 
			\\
			&
			\sup_{t \in [A+T,A+2T]u^{-2/\kappa}} \overline{X}(t^{\kappa/\alpha}) > u
		\end{aligned}
	}
	\leq \; F  \left( T^2 \exp \left( - G \sqrt{T^\kappa} \right) + \sqrt{T} \right) \Psi(u) \;  
\end{align*}
for all $ A> 0 $, $T\geq 1$ and any $ u \geq u_0 = (4\delta)^{-\kappa/2} (A+2T)^{\kappa/2} $, i.e. $ (A+2T)u^{-2/\kappa} \leq 4\delta$.
\\
\EEL

\BEL\label{betacut}
\begin{enumerate}[(i)]
	\item If  $\alpha \leq \beta$, then for all small enough $\delta>0$
	\bqn{ \lim_{M \to \infty} \underset{u \to \infty}{\limsup}\; \frac{1}{u^{2/\kappa-2/\hat{\beta}}\Psi(u)}\pk*{ \underset{t \in [Mu^{-2/\hat{\beta}},\delta]}{\sup} X(t^{\kappa/\alpha})>u}  &=& 0;
	}
	
	\item If  $\alpha>\beta$, then for all small enough $\delta>0$  
	\bqn{ \lim_{M \to \infty} \underset{u \to \infty}{\limsup}\; \frac{1}{\Psi(u)} \pk*{ \underset{t \in [Mu^{-2/\hat{\beta}},\delta]}{\sup} X(t^{\kappa/\alpha})>u} &=& 0.
	}
\end{enumerate}
\EEL

\BEL\label{gammacut}
\begin{enumerate}[i)] 
	\item  If $\alpha \leq 2\gamma$, then for all small enough $\delta>0$
	\bqn{ \limit{M}\underset{u \to \infty}{\limsup}\; \frac{1}{u^{2/\kappa-1/\hat{\gamma}}\Psi(u)}\pk*{\underset{t \in [Mu^{-1/\hat{\gamma}},\delta]}{\sup}\left( \overline{X}(t^{\kappa/\alpha})-u-dt^{\hat{\gamma}}\right)>0} = 0;
	} 
	
	\item  If $\alpha > 2\gamma$, then for all small enough $\delta>0$
	\bqn{ \limit{M}\underset{u \to \infty}{\limsup}\; \frac{1}{\Psi(u)}\pk*{\underset{t \in [Mu^{-1/\hat{\gamma}},\delta]}{\sup}
			\left(
			\overline{X}(t^{\kappa/\alpha})-u-dt^{\hat{\gamma}}
			\right)
			>0} = 0. 
	}
\end{enumerate}
\EEL

\BEL\label{kappacut}
If $\alpha>\kappa$, then for $\epsilon \in [0,2/\kappa-2/\alpha)$ and $L_u=Lu^{-2/\alpha-\epsilon}$ we have   
\bqn{ \limit{\cEE{M}}\underset{u \to \infty}{\limsup}\;\frac{1}{\Psi(u)u^{\frac{\epsilon\alpha}{\alpha-\kappa}}}\pk*{\underset{t \in [\cEE{M} u^{-2/\kappa+\frac{\epsilon\alpha}{\alpha-\kappa}},\delta]}{\sup} \underset{s \in [0,1]}{\inf} \overline{{X}}((t+\widehat{L}_{u,t}s)^{\cEE{\kappa/\alpha}})>u} = 0.
}
\EEL

\BEL\label{smallcut}
For each $p \in (0,2/\kappa)$ we have 
\begin{align*}
\lim_{M \to \infty}
\underset{u \to \infty}{\limsup}\; \frac{1}{u^{p}\Psi\left(u\right)}\pk*{\underset{t \in [0,\frac{1}{M}u^{-2/\kappa+p}]}{\sup} \overline{X}(t^{\kappa/\alpha})>u} = 0.
\end{align*}
\EEL
For the proofs of Lemmas \ref{Lem25}, \ref{Lem27}, \ref{betacut}, \ref{gammacut}, \ref{smallcut} see \citet*{novikov2025sojourn}. Remark that for the cases of \Cref{mainth} where $p>0$ (that is, cases (i) and (ii)) we have
\begin{align}
\label{lup}
L_u = 
L u^{-\frac{2+p(\alpha-\kappa)}{\alpha}},\; 
p
=
\begin{cases}
	&
	\min\left(2/\kappa-2/\hat{\beta},2/\kappa-1/\hat{\gamma},\frac{\epsilon \alpha}{\alpha-\kappa} \right),\,\alpha>\kappa,\\
	&
	\min\left(2/\kappa-2/\hat{\beta},2/\kappa-1/\hat{\gamma} \right),\,\alpha \leq \kappa,
\end{cases}
\end{align}

and therefore, Lemmas \ref{betacut}-\ref{kappacut} imply
\begin{align}
\label{largecut}
\lim_{M \to \infty}
\underset{u \to \infty}{\limsup}\; \frac{1}{u^{p}\Psi\left(u\right)}\pk*{\underset{t \in [Mu^{-2/\kappa+p},\delta]}{\sup}
	\underset{s \in [0,1]}{\inf}
	\left( X((t+\widehat{L}_{u,t}s)^{\cEE{\kappa/\alpha}})
	-d(t+\widehat{L}_{u,t}s)^{\hat{\gamma}}
	\right)>u} = 0.
\end{align}

For the cases of \Cref{mainth} where $p=0$ (that is, case
(iii)) Lemmas \ref{betacut}-\ref{kappacut} imply
\begin{align}
\label{largecut2}
\lim_{M \to \infty}
\underset{u \to \infty}{\limsup}\; \frac{1}{\Psi\left(u\right)}\pk*{\underset{t \in [Mu^{\min(-2/\kappa,-2/\hat{\beta},-1/\hat{\gamma})},\delta]}{\sup}
	\underset{s \in [0,1]}{\inf}
	\left( X((t+\widehat{L}_{u,t}s)^{\cEE{\kappa/\alpha}})
	-d(t+\widehat{L}_{u,t}s)^{\hat{\gamma}}
	\right)>u} = 0.
\end{align}

\section{Proofs}
\begin{proof}[Proof of \Cref{the-weak-conv}]
	Recall the notation 
	\[
		\omega_f(\delta):=
		\sup_{\vk{t},\vk{s} \in E,\,\norm{\vk{t}-\vk{s}} \leq \delta} |f(\vk{t})-f(\vk{s})|.
	\]
	
	Recall the definition of $C(E)$ as the Banach space of continuous functions on $E$.	As common in the literature we call a family $\{f_{u,j}(\vk{s}),j \in S_u\}$  of functions from $C(E)$ {tight}, if there exist $u_0>0$ and $M>0$ such that
	$$
	\sup_{u \geq u_0,j \in S_u} |f_{u,j}(\vk{0})| \leq M ,   \quad 
	\limsup_{\cEE{\delta \downarrow 0}}\sup_{u \geq u_0} \sup_{j \in S_u} \omega_{f_{u,j}}(\delta) =0.
	$$ 
	
	Also as in the literature, we call a family of probability measures $\{P_{u,j},j \in S_u\}$ on $C(E)$ tight 
	if there exists $u_0>0$ such that 
	$$
	\lim_{a \to \infty} \sup_{u\geq u_0} \sup_{j\in S_u} P_{u,j}\{f : |f(\vk{0})|\geq a\}=0,  \quad 
	\lim_{\cEE{\delta \downarrow 0}}\sup_{u \geq u_0} \sup_{j\in S_u} P_{u,j}\{f : \omega_f(\delta)\geq \epsilon\}=0
	$$
	for all $\epsilon>0$.

	Remark that if \eqref{con-uni-con} is not fulfilled then there exists a sequence of pairs $(u_1,j_1),(u_2,j_2),...$ and $\epsilon>0$ such that \ksn{$u_k$ form an increasing sequence}, $u_k \to \infty$ as $k\to \infty$ and for all integers $k\ge 1$
	
	\begin{align}\label{con-uni-contr}
	\ABs{ \frac{\pk*{   \cEE{\Gamma}(\xi_{u_k,j_k})>g_{u_k,j_k} } } {\Psi(g_{u_k,j_k})} -\mathcal{H}^{\Gamma, h_{u_k,j_k}}_{\zeta_{u_k,j_k}}} > \epsilon ,\,
	\cEE{\mathcal{H}^{\Gamma, h}_{\zeta}\coloneqq 
		\E{ e^{ 			\Gamma(\sqrt{2}\zeta-\mathrm{Var}(\zeta)-h)}}.}
	\end{align} 
	
	Therefore, it is enough to prove \Cref{the-weak-conv} only for the case when $u \in \{u_1,u_2,u_3,..\}$ and $S_{u_k}=\{j_k\}$. 
	
	Moreover, since $h_{u_k,j_k}$ form a tight family of continuous functions by \Cref{C1}, by passing to a subsequence we can assume that there exists a continuous function $h$ such that $h_{u_k,j_k}(t) \to h(t)$ uniformly in $t$ as $k \to \infty$.
	
	Since $\zeta_{u_k,j_k}$ form a tight
	family by \Cref{C2}-\Cref{C3}, we can also assume \cEE{the weak convergence} $\zeta_{u_k,j_k}(t) \Rightarrow
	\zeta(t)$ as $k \to \infty$ for some centered Gaussian process $\zeta$ with a.s. continuous paths. Therefore, the conditions \Cref{C1}-\Cref{C2} also work for the setting when $\zeta_{u_k,j_k}$ and $h_{u_k,j_k}$ are replaced by $\zeta$ and $h$, but then by \citet*[Lemma 4.1]{debicki2023sojourn} we have
	\begin{align}\label{con-uni-con-easy}
	\lim_{k\to\IF} \ABs{ \frac{\pk*{  \Gamma(\xi_{u_k,j_k})>g_{u_k,j_k} } } {\Psi(g_{u_k,j_k})} -\mathcal{H}^{\Gamma, h}_{\zeta} } = 0. \quad 
	\end{align}
	\ksn{Formally speaking, \citet*[Lemma 4.1]{debicki2023sojourn} is slightly too general for this purpose, however, \citet*[Theorem 2.2]{Uniform2016} has slightly different assumptions (in assumption (D2) from \citet*[Theorem 2.2]{Uniform2016}
		non-standardized Gaussian processs $\xi_{u,\tau_u}$ are considered, while in \Cref{C21} we consider normalized standardized Gaussian processes $\overline{\xi}_{u,j}$), therefore we prefer to apply \citet*[Lemma 4.1]{debicki2023sojourn}. Also remark that our \Cref{f2} is slightly weaker than condition F3 from \citet*[Lemma 4.1]{debicki2023sojourn}, but one can verify that the proof works remains applicable under \Cref{f2} as well.}
	
	It only remains to check \eqref{bddconst} and to prove that
	\begin{align}
	\label{constconv}
	\mathcal{H}^{\Gamma, h_{u_k,j_k}}_{\zeta_{u_k,j_k}}
	-
	\mathcal{H}^{\Gamma, h}_{\zeta}
	\to 0
	\end{align}
	as $k\to\infty$. Remark that 
	\\
	$$
	\sqrt{2}\zeta_{u_k,j_k}-\mathrm{Var}(\zeta_{u_k,j_k})-h_{u_k,j_k} \Rightarrow
	\sqrt{2}\zeta-\mathrm{Var}(\zeta)-h
	$$
	as $k \to \infty$, and therefore, denoting 
	\[
	\mathbb{A}\coloneq\{x: \pk{\Gamma(\sqrt{2}\zeta-\mathrm{Var}(\zeta)-h)>x} \text{ is discontinuous at }x\},
	\] for all $x \in \mathbb{R}\backslash\mathbb{A}$ we obtain		
	\[
	\pk{\Gamma(\sqrt{2}\zeta_{u_k,j_k}-\mathrm{Var}(\zeta_{u_k,j_k})-h_{u_k,j_k})>x}
	\to 
	\pk{\Gamma(\sqrt{2}\zeta-\mathrm{Var}(\zeta)-h)>x}
	\]
	as $k \to \infty$. Since $\mathbb{A}$ is at most countable, and 
	\begin{align*}
	\mathcal{H}^{\Gamma,h_{u_k,j_k}}_{\zeta_{u_k,j_k}}
	=\int_{-\infty}^{\infty} e^x \pk{\Gamma(\sqrt{2}\zeta_{u_k,j_k}-\mathrm{Var}(\zeta_{u_k,j_k})-h_{u_k,j_k})>x},
	\end{align*}
	by the Lebesgue dominated convergence theorem in order to prove
	\eqref{constconv} it remains to prove the existence of an integrable majorant independent of $u \geq u_0$ and $j \in S_u$
	for
	\\ $e^x\pk{\Gamma(\sqrt{2}\zeta_{u,j}-\mathrm{Var}(\zeta_{u,j})-h_{u,j})>x}$ (such a majorant will be in particular a majorant for 
	\\
	$e^x \pk{\Gamma(\sqrt{2}\zeta_{u_k,j_k}-\mathrm{Var}(\zeta_{u_k,j_k})-h_{u_k,j_k})>x}$). Remark that since
	\[
	\mathcal{H}^{\Gamma,h_{u,j}}_{\zeta_{u,j}}
	=\int_{-\infty}^{\infty} e^x \pk{\Gamma(\sqrt{2}\zeta_{u,j}-\mathrm{Var}(\zeta_{u,j})-h_{u,j})>x} dx,
	\] 
	the existence of an integrable majorant for $e^x\pk{\Gamma(\sqrt{2}\zeta_{u,j}-\mathrm{Var}(\zeta_{u,j})-h_{u,j})>x}$ implies \eqref{bddconst} as well.
	
	By \Cref{f2} and the Piterbarg inequality \citet*[Theorem 8.1]{Pit96}, denoting \[
	M_2 \coloneq -\inf_{t \in E,j \in S_u, u \geq u_0}h_{u,j}(t),
	\] and fixing $\epsilon \in (0,1)$, we obtain that there exists $x_0>0$ such that for all $x \geq x_0$, $u \geq u_0$ and $j \in S_u$ 
	\begin{align*}
	&
	\pk{\Gamma(\sqrt{2}\zeta_{u,j}-\mathrm{Var}(\zeta_{u,j})-h_{u,j})>x}
	\\
	&
	\leq
	\pk{ \sup_{t \in E}(\sqrt{2}\zeta_{u,j}(t)-\mathrm{Var}(\zeta_{u,j}(t))-h_{u,j}(t))>x/c_1}
	\\
	&
	\leq
	\pk{ \sup_{t \in E} \sqrt{2}\zeta_{u,j}(t)>x/c_1-M_2}
	\\
	&
	\leq
	\pk{ \sup_{t \in E} \zeta_{u,j}(t)>\frac{x}{2c_1} }
	\leq \Psi\left(\frac{(1-\epsilon)x}{2 c_1 \sigma}  \right),
	\end{align*}
	\ksn{where $\sigma^2$ is defined by \eqref{bddvar}}.
	
	Therefore, $e^x \Psi\left(\frac{(1-\epsilon)x}{2 c_1 \sigma}  \right) \mathbb{I}(x \geq x_0)+e^x \mathbb{I}(x <x_0)$ is an integrable majorant for
	\\
	$e^x\pk{\Gamma(\sqrt{2}\zeta_{u,j}-\mathrm{Var}(\zeta_{u,j})-h_{u,j})>x}$, $u \geq u_0,\,j \in S_u$, finishing the proof.
\end{proof}

\begin{proof}[Proof of Corollary \ref{the-weak-conv-cor}] Denote 
\[
	\widetilde{\xi}_{u,j}(t,s):=\xi_{u,j}(t+s),\,\widetilde{\zeta}_{u,j}(t,s):=\zeta_{u,j}(t+s),\,\widetilde{h}_{u,j}(t,s):=h_{u,j}(t+s).
\]
Since $\xi_{u,j},\zeta_{u,j}$ and $h_{u,j}$ satisfy assumptions C1-C4 with $E=[0,T+L]$, $\widetilde{\xi}_{u,j},\widetilde{\zeta}_{u,j}$ and $\widetilde{h}_{u,j}$ satisfy assumptions C1-C4 as well with $E=[0,T] \times [0,L]$. Therefore, by \Cref{the-weak-conv} with 
\[
	\Gamma(f)=\sup_{t \in [0,T]} \inf_{s \in [0,L]} f(t,s)
\]
it holds that
\begin{align}
\label{conv-bivar}
\lim_{u\to\IF}\sup_{j\in S_u} &\Bigg| \frac{\pk*{ \sup_{t \in [0,T]}\inf_{s \in [0,L]}\widetilde{\xi}_{u,j}(t,s)>g_{u,j} } }{\Psi(g_{u,j})} \notag
\\
&\quad-\E{ e^{ 			\sup_{t \in [0,T]}\inf_{s \in [0,L]}(\sqrt{2}\widetilde{\zeta}_{u,j}(t,s)-\mathrm{Var}(\widetilde{\zeta}_{u,j}(t,s))-\widetilde{h}_{u,j}(t,s))}}\Bigg|=0.  
\end{align}
In order to finish the proof, remark that \eqref{conv-bivar} is equivalent to \eqref{con-uni-con-cor}.
\end{proof}

\cEE{For the following proofs recall the definition \eqref{betahat} of $\hat{\gamma}$, the definition \eqref{Zhat} of $\widehat X$ and $\widehat Y$ and the definition \eqref{lut} of $\widehat{L}_{u,t}$. Next, we obtain} 
\begin{align}
\label{borell_prelim}
&\pk*{\underset{t \in [0,T]}{\sup}\; \underset{s \in [0,L_u]}{\inf} \left(X(t+s)-d(t+s)^{\gamma}\right) > u }\notag\\
&=
\pk*{\underset{t \in [0,T]}{\sup}\; \underset{s \in [0,L_u]}{\inf}
	\left( \widehat{X}((t+s)^{\alpha/\kappa})-d(t+s)^{\gamma}
	\right) > u }\notag\\
&
=\pk*{\underset{t \in [0,T]}{\sup}\; \underset{v \in [t^{\alpha/\kappa},(t+L_u)^{\alpha/\kappa}]}{\inf}
	\left( \widehat{X}(v)-dv^{\hat{\gamma}}
	\right) > u }
\notag
\\
&=
\pk*{\underset{t \in [0,T]}{\sup}\; \underset{s \in [0,1]}{\inf}
	\left( \widehat{X}(t^{\alpha/\kappa}+s((t+L_u)^{\alpha/\kappa}-t^{\alpha/\kappa}))-d(t^{\alpha/\kappa}+s((t+L_u)^{\alpha/\kappa}-t^{\alpha/\kappa}))^{\hat{\gamma}} 
	\right)
	> u }\notag\\
&=
\pk*{\underset{t \in [0,T^{\alpha/\kappa}]}{\sup}\; \underset{s \in [0,1]}{\inf}
	\left( \widehat{X}(t+\widehat{L}_{u,t}s)-d(t+\widehat{L}_{u,t}s)^{\hat{\gamma}}
	\right)
	> u }.
\end{align}
However, for arbitrary $\delta>0$ for sufficiently large $u$ by Borell-TIS inequality (see, e.g., \citet*{Pit96})
\begin{align}
&\pk*{\underset{t \in [\delta,T^{\alpha/\kappa}]}{\sup}\; \underset{s \in [0,1]}{\inf}
	\left( \widehat{X}(t+\widehat{L}_{u,t}s)-d(t+\widehat{L}_{u,t}s)^{\hat{\gamma}}
	\right)
	> u } \notag
\\
&\leq
\pk*{\underset{t \in [\delta,T^{\alpha/\kappa}]}{\sup}\; \underset{s \in [0,1]}{\inf}
	 \widehat{X}(t+\widehat{L}_{u,t}s)
	> u } \notag
\\	
&\leq
\pk*{\underset{t \in [\delta,T^{\alpha/\kappa}]}{\sup}\; \widehat{X}(t)
	> u } 
	=
	\pk{\sup_{t \in [\delta^{\kappa/\alpha},T]} \widehat{X}(t) > u}
	\leq 
	 e^{-\frac{(u-m)^2}{\sigma_\delta^2}}, 
\end{align}
where $\sigma_\delta^2 = \sup_{t \in [\delta^{\kappa/\alpha},T]} \mathrm{Var}(X(t)) < 1$ and $m = \E{\sup_{t \in [0,T]}X(t)}$.

Therefore, by \eqref{borell_prelim} as $u \to \infty$  

\begin{align}
	\label{borell}
	&
	\pk*{\underset{t \in [0,T]}{\sup}\; \underset{s \in [0,L_u]}{\inf} \left(X(t+s)-d(t+s)^{\gamma}\right) > u } \notag
	\\
	&
	=
	\pk*{\underset{t \in [0,\delta]}{\sup}\; \underset{s \in [0,1]}{\inf}
		\left( \widehat{X}(t+\widehat{L}_{u,t}s)-d(t+\widehat{L}_{u,t}s)^{\hat{\gamma}}
		\right)
		> u }+o(\Psi(u)).
\end{align}

\cEE{By \Cref{LS2}, as shown in \citet*[Corollary
	2.1]{tabis}, it follows that both constants (recall that $V_{\widehat Y}$ is the variogram of $\widehat Y$)}

\begin{align}
\label{varbounds}
c_1 = \underset{x \in [0,1]}{\inf}\frac{V_{\widehat{Y}}(1,x)}{|1-x|^{\kappa}},\quad c_2 = \underset{x \in [0,1]}{\sup}\frac{V_{\widehat{Y}}(1,x)}{|1-x|^{\kappa}}
\end{align}
are positive and finite. 
Choose  $\delta>0$ sufficiently small such that  for all $t,s \in [0,4\delta]$ 
$$1-r_{\widehat{X}}(t,s)\in \left[\frac{a}{2} V_{\widehat{Y}}(t,s),2a V_{\widehat{Y}}(t,s)\right]$$
and $4\delta\leq (2ac_2)^{1/\kappa}$.

However, by \eqref{varbounds} if $t \geq s$ (and analogously for $t<s$), then
\begin{align}
&
V_{\widehat{Y}}(t,s)
=t^{\kappa}V_{\widehat{Y}}(1,s/t) \leq c_2 t^\kappa |1-s/t|^\kappa = c_2 |t-s|^\kappa, \notag
\\
&V_{\widehat{Y}}(t,s)
=t^{\kappa}V_{\widehat{Y}}(1,s/t) \geq c_1 t^\kappa |1-s/t|^\kappa = c_1 |t-s|^\kappa.
\end{align}

Consequently, we obtain 
\begin{align}
\label{bddcorr}
1-r_{\widehat{X}}(t,s) \in 
\left[\frac{ac_1}{2} |t-s|^{\kappa},2ac_2 |t-s|^{\kappa}\right] \text{ for all } t,s \in [0,4\delta].
\end{align}

Moreover, by \Cref{LS1} and \Cref{LS2}
\begin{align}
\label{localBM}
V_{\widehat{Y}}(t,s)
=t^{\kappa}V_{\widehat{Y}}(1,s/t)
\sim c_{\widehat{Y}}t^{\kappa}|1-s/t|^{\kappa}
=c_{\widehat{Y}}|t-s|^{\kappa}
\end{align}
uniformly in $t>s>0$ as $t/s \to 1$.

\begin{proof}[Proof of Lemma \ref{kappacut}]
	Denote for convenience $Z(t)\coloneq \overline{X}(t^{\kappa/\alpha})$. Recall that by \eqref{lut}
	$$t+\widehat{L}_{u,t}s = t+s((t^{\kappa/\alpha}+L_u)^{\alpha/\kappa}-t) \geq t + (\alpha/\kappa)L_u t^{1-\kappa/\alpha} s.$$ 
	\def\wMM{\cEE{\widetilde{M}_{u,k}}}
	Consequently,
	\begin{align}
	\label{chain_prelim}
	&\pk*{\underset{t \in [\cEE{M} u^{-2/\kappa+\frac{\epsilon\alpha}{\alpha-\kappa}},\delta]}{\sup} \underset{s \in [0,1]}{\inf} \overline{{X}}((t+\widehat{L}_{u,t}s)^{\cEE{\kappa/\alpha}})>u}
	\notag
	\\
	&
	=\pk*{\underset{t \in [\cEE{M} u^{-2/\kappa+\frac{\epsilon\alpha}{\alpha-\kappa}},\delta]}{\sup} \underset{s \in [0,1]}{\inf} Z(t+\widehat{L}_{u,t}s)>u} \notag
	\\
	& \leq \pk*{\underset{t \in [Mu^{-2/\kappa+\frac{\epsilon\alpha}{\alpha-\kappa}},\delta]}{\sup}\underset{s \in [0,1]}{\inf}Z\left(t+\left(\alpha/\kappa\right)L_u t^{1-\kappa/\alpha} s\right)>u}  
	\end{align}
	
	However, fixing $M$ large enough and setting $\wMM\coloneq \frac{[Mu^{\frac{\epsilon \alpha}{\alpha-\kappa}} ]+k}{u^{\frac{\epsilon\alpha}{\alpha-\kappa}}},\Delta=u^{-2/\kappa}$ for
	$u$ large enough for some $C>0$ independent of $u$ and $M$ we have 
	\begin{align}
	\label{chain}
	& \pk*{\underset{t \in [Mu^{-2/\kappa+\frac{\epsilon\alpha}{\alpha-\kappa}},\delta]}{\sup}\underset{s \in [0,1]}{\inf}Z\left(t+\left(\alpha/\kappa\right)L_u t^{1-\kappa/\alpha} s\right)>u}
	\notag  
	\\& \leq
	\sum\limits_{k=[Mu^{\frac{\epsilon\alpha}{\alpha-\kappa}}]}^{[\delta/\Delta]+1}\pk*{\underset{t \in [k\Delta,\left(k+1\right)\Delta]}{\sup}\underset{s \in [0,1]}{\inf} Z\left(t+ \left(\alpha/\kappa\right)L_u t^{1-\kappa/\alpha} s\right)>u} 
	\notag 
	\\&\leq
	\sum\limits_{k=[Mu^{\frac{\epsilon\alpha}{\alpha-\kappa}}]}^{[\delta/\Delta]+1} \pk*{ Z\left(\left(k+1\right)\Delta\right)>u,\;Z\left(k\Delta+ \left(\alpha/\kappa\right)L u^{-2/\kappa} \left(\frac{k}{u^{\frac{\epsilon\alpha}{\alpha-\kappa}}}\right)^{\frac{\alpha-\kappa}{\alpha}} \right) > u } 
	\notag
	\\&\leq
	\sum\limits_{k=0}^{\infty}C \Psi\left(u\right)
	\exp\left(-\frac{\left|-1+ \left(\alpha/\kappa\right)L  \wMM^{\frac{\alpha-\kappa}{\alpha}}\right|^{\kappa}}{2C}\right)
	\notag
	\\&\leq
	C\Psi\left(u\right)u^{\frac{\epsilon\alpha}{\alpha-\kappa}}
	\sum\limits_{k=0}^{\infty}	
	\exp\left(-(M+(k+1)u^{-\frac{\epsilon \alpha}{\alpha- \kappa} })^{\frac{\kappa(\alpha-\kappa)}{2 \alpha}} \right) \cdot\left((k+1)u^{-\frac{\epsilon\alpha}{\alpha-\kappa}}-ku^{-\frac{\epsilon\alpha}{\alpha-\kappa}}\right)
	\notag
	\\& \leq 
	C \Psi\left(u\right)u^{\frac{\epsilon\alpha}{\alpha-\kappa}}
	\int\limits_{0}^{\infty}\exp\left(
	-\left(M+x\right)^{\frac{\kappa\left(\alpha-\kappa\right)}{2\alpha}}
	\right) dx.
	\end{align}
	The second inequality of \eqref{chain} follows from the fact that if $M$ is large enough then for $u$ large enough and $t \in [k \Delta,(k+1)\Delta]$ it holds that 
	\[
	(k+1)\Delta, k\Delta+ \left(\alpha/\kappa\right)L u^{-2/\kappa} \left(\frac{k}{u^{\frac{\epsilon\alpha}{\alpha-\kappa}}}\right)^{\frac{\alpha-\kappa}{\alpha}} \in [t,t+(\alpha/\kappa)L_u t^{1-\kappa/\alpha}].
	\]
	Indeed,
	\begin{align*}
	&
	t+(\alpha/\kappa)L_u t^{1-\kappa/\alpha}\geq
	k\Delta+(\alpha/\kappa)L_u (k \Delta)^{1-\kappa/\alpha}
	\\
	&
	=
	k\Delta+ \left(\alpha/\kappa\right)L u^{-2/\kappa} \left(\frac{k}{u^{\frac{\epsilon\alpha}{\alpha-\kappa}}}\right)^{\frac{\alpha-\kappa}{\alpha}}
	\\
	&
	\geq 
	k\Delta+ \left(\alpha/\kappa\right)L u^{-2/\kappa} \left(\frac{M}{2}\right)^{\frac{\alpha-\kappa}{\alpha}}
	\geq (k+1)\Delta
	\geq t.
	\end{align*}
	In order to check the third inequality of \eqref{chain}, remark that under the conditions of \Cref{Lem25}
	\begin{align*}
		&
		\pk*{ 
			\begin{aligned}
				&
				Z((A+T) u^{-2/\kappa}) > u ,
				\\&
				Z(t_0 u^{-2/\kappa}) > u 
			\end{aligned}
		}
		\\
		&
		\leq
		\pk*{ 
			\begin{aligned}
				&
				\sup_{t \in [A,A+T]u^{-2/\kappa}} Z(t) > u ,
				\\&
				\sup_{t \in [t_0,t_0+T]u^{-2/\kappa}} Z(t) > u 
			\end{aligned}
		}
		=
		\pk*{ 
			\begin{aligned}
				&
				\sup_{t \in [A,A+T]u^{-2/\kappa}} \overline{X}(t^{\kappa/\alpha}) > u ,
				\\&
				\sup_{t \in [t_0,t_0+T]u^{-2/\kappa}} \overline{X}(t^{\kappa/\alpha}) > u 
			\end{aligned}
		}.
	\end{align*}
	Therefore, the third inequality of \eqref{chain} follows from \Cref{Lem25} with
	\\ $A=k,T=1,t_0=k+(\alpha/\kappa)L\left(\frac{k}{u^{\frac{\epsilon\alpha}{\alpha-\kappa}}}\right)^{\frac{\alpha-\kappa}{\alpha}}$, because for $k \geq [Mu^{\frac{\epsilon \alpha}{\alpha-\kappa} }]$ it holds that
	\[
	k+ \left(\alpha/\kappa\right)L
	\left(\frac{k}{u^{\frac{\epsilon\alpha}{\alpha-\kappa}}}\right)^{\frac{\alpha-\kappa}{\alpha}}-(k+1)
	=
	-1+ \left(\alpha/\kappa\right)L  \widetilde{M}_{u,k-[Mu^{\frac{\epsilon \alpha}{\alpha-\kappa} }]}^{\frac{\alpha-\kappa}{\alpha}}.
	\]
	The fourth inequality of \eqref{chain} follows from the fact that if $M$ is large enough, then for $u$ large enough for all integers $k \geq 0$ it holds that
	\[
	\frac{\left|-1+ \left(\alpha/\kappa\right)L  \widetilde{M}_{u,k}^{\frac{\alpha-\kappa}{\alpha} }\right|^{\kappa}}{2C}
	\geq 
	(M+(k+1)u^{-\frac{\epsilon \alpha}{\alpha-\kappa}} )^{\frac{\kappa(\alpha-\kappa)}{2\alpha}}. 
	\]
	
	Therefore, we obtain
	\begin{align*}
	&\underset{u \to \infty}{\limsup}\;\frac{1}{\Psi\left(u\right)u^{\frac{\epsilon\alpha}{\alpha-\kappa}}}\pk*{\underset{t \in [Mu^{-2/\kappa+\frac{\epsilon\alpha}{\alpha-\kappa}},\delta]}{\sup} \underset{s \in [0,1]}{\inf} Z\left(t+\widehat{L}_{u,t}s\right)-u>0} \\
	& \leq  C
	\int\limits_{0}^{\infty}\exp\left(
	-\left(M+x\right)^{\frac{\kappa\left(\alpha-\kappa\right)}{2\alpha}}
	\right) dx \to 0
	\end{align*}
	as $M \to \infty$ establishing the proof.
\end{proof}

\begin{proof}[Proof of Theorem \ref{last}]:
	Let $\widetilde{X}(t),\;t \in [0,T]$ be a centered Gaussian process with covariance function
	\\
	$R_{\widetilde{X}}(t,s)=\exp\left(-V_{\widehat{Y}}(t,s)\right)$. Define $X_u(t)\coloneq\widetilde{X}(tu^{-2/\kappa})$.	
	
	By Corollary \ref{the-weak-conv-cor} with $S_u=\{1\}$, $g_{u,1}=u$,
	$\xi_{u,1}(t)=X_u(t)$ and $\zeta_{u,1}(t)=\widehat{Y}(t)$ for each $T>0$ we have
	\begin{align} 
	\label{defut_prelim}
	\lim_{u \to \infty}
	\left|\frac{1}{\Psi(u)}
	\pk*{\sup_{t \in [0,T]}\inf_{s\in [0,L]} X_u(t+s)
		>u
	}-\mathcal{H}^{\mathrm{par}}_{\widehat{Y},T,L}\right| =0.
	\end{align}
	
	Therefore, we can find $u(T)>0$ such that
	\begin{align} 
	\label{defut}
	\left|\frac{1}{\Psi(u(T))}
	\pk*{\sup_{t \in [0,T]}\inf_{s\in [0,L]} X_{u(T)}(t+s)
		>u(T)
	}-\mathcal{H}^{\mathrm{par}}_{\widehat{Y},T,L}\right| \leq 1.
	\end{align}
	Moreover, without loss of generality we can assume that $u(T)$ is so large that $u(T)^{-2/\kappa}T \to 0$ as $T \to \infty$.
	
	Now fix $S \in \mathbb{Z}_{>0}$, let $n \in \mathbb{Z}_{>0}$ and denote
	\begin{align*}
	\Sigma\Sigma_{n,S}\coloneq
	\sum\limits_{k=n}^{n^2-1}
	\sum\limits_{l=k+1}^{n^2-1}
	\pk*{\underset{[kS,(k+1)S]}{\sup} \;X_{u(n^2 S)}(t)>u(n^2 S),\;
		\underset{[lS,(l+1)S]}{\sup} \;X_{u(n^2 S)}(t)>u(n^2 S).
	}
	\end{align*}
	With this notation we have 
	\begin{align}
	\label{ubconst}
	&\pk*{\sup_{t \in [0,n^2 S]}\inf_{s\in [0,L]} X_{u(n^2 S)}(t+s)
		>u(n^2 S)
	} \notag
	\\ &\leq 
	\pk*{\sup_{t \in [0,n S]}\inf_{s\in [0,L]} X_{u(n^2 S)}(t+s)
		>u(n^2 S)
	}+
	\pk*{\sup_{t \in [nS,n^2 S]}\inf_{s\in [0,L]} X_{u(n^2 S)}(t+s)
		>u(n^2 S)
	} \notag
	\\&
	\leq \pk*{\sup_{t \in [0,n S]}\inf_{s\in [0,L]} X_{u(n^2 S)}(t+s)
		>u(n^2 S)
	}+
	\sum\limits_{k=n}^{n^2-1}\pk*{\sup_{t \in [kS,(k+1) S]}\inf_{s\in [0,L]} X_{u(n^2 S)}(t+s)
		>u(n^2 S)
	}.
	\end{align}
	Moreover,
	\begin{align}
	\label{lbconst}
	\pk*{\sup_{t \in [0,n^2 S]}\inf_{s\in [0,L]} X_{u(n^2 S)}(t+s)
		>u(n^2 S)
	}
	&\geq \pk*{\sup_{t \in [nS,n^2 S]}\inf_{s\in [0,L]} X_{u(n^2 S)}(t+s)
		>u(n^2 S)
	}
	\notag
	\\
	&\geq \sum\limits_{k=n}^{n^2-1}\pk*{\sup_{t \in [kS,(k+1) S]}\inf_{s\in [0,L]} X_{u(n^2 S)}(t+s)
		>u(n^2 S) \notag
	}
	\\
	&\quad
	-\Sigma\Sigma_{n,S}.
	\end{align}
	Let 
	$ \underset{n \to \infty}{\limsup} \frac{\Sigma\Sigma_{n,S}}{n^2 S \Psi(u(n^2 S))} \eqcolon f(S).
	$ Then by \Cref{Lem25} and
	\Cref{Lem27} we have 
	$f(S) \to 0$ as $S \to \infty$.
	
	Moreover, for $t,s \geq 0$ for $c_2$ defined in \eqref{varbounds}
	\begin{align*}
	\mathrm{Cov}(X_u(t),X_u(s))
	=\exp(-V_{\widehat{Y}}(t u^{-2/\kappa},s u^{-2/\kappa}))
	=
	\exp(-u^{-2} V_{\widehat{Y}}(t ,s ))
	\geq 
	\exp(-u^{-2} c_2 V_{B_\kappa}(t ,s )),
	\end{align*}
	hence, for a fixed $S>0$, denoting by $\overline{X}_2(t)$ a stationary centered Gaussian process with covariance function $R_{\overline{X}_2}(t,s)=\exp(-|t-s|^\kappa)$,
	using Slepian's inequality for the first inequality and \Cref{con-uni-con} for the last asymptotics, we obtain 
	\begin{align*}
	\pk*{\underset{t \in [0,nS]}{\sup} \;X_{u(n^2 S)}(t)>u(n^2 S)}
	&\leq&
	\pk{\sup_{t \in [0,nS]}\overline{X}_2(t u(n^2 S)^{-2/\kappa})>u(n^2 S)}
	\notag
	\\
	&\leq& 
	n \pk{\sup_{t \in [0,S]}\overline{X}_2(t u(n^2 S)^{-2/\kappa})>u(n^2 S)}.
	\end{align*}
	Consequently, invoking \Cref{con-uni-con} yields
	\begin{align}
	\label{startbound}
	\pk*{\underset{t \in [0,nS]}{\sup} \;X_{u(n^2 S)}(t)>u(n^2 S)}=O(n \Psi(u(n^2 S)))
	\end{align}
	as $n \to \infty$.
	
	Therefore, fixing $S$ and letting $n \to \infty$ in (\ref{ubconst}) and (\ref{lbconst}), divided by $n^2 S \Psi(u(n^2 S))$, using (\ref{defut}) for the left-hand side and
	using further \eqref{startbound} and \Cref{the-weak-conv-cor} for the right-hand sides we obtain
	\begin{align*}
	\frac{
		\mathcal{H}^{\mathrm{par}}_{B_{\kappa}(c_{\widehat{Y}}^{1/\kappa}\cdot),S,L}
	}
	{S}-f(S)
	\leq \underset{n\to \IF}{\liminf}
	\frac{
		\mathcal{H}^{\mathrm{par}}_{\widehat{Y},n^2 S,L}
	}
	{n^2 S}
	\leq
	\underset{n\to \IF}{\limsup}\frac{
		\mathcal{H}^{\mathrm{par}}_{\widehat{Y},n^2 S,L}
	}{n^2 S}
	\leq 
	\frac{\mathcal{H}^{\mathrm{par}}_{B_{\kappa}(c_{\widehat{Y}}^{1/\kappa}\cdot),S,L}
	}
	{S},
	\end{align*}
	which implies
	\begin{align*}
	c_{\widehat{Y}}^{1/\kappa}
	\frac{
		\mathcal{H}^{\mathrm{par}}_{\kappa,c_{\widehat{Y}}^{1/\kappa}S,c_{\widehat{Y}}^{1/\kappa}L}
	}
	{c_{\widehat{Y}}^{1/\kappa} S}-f(S)
	\leq \underset{n\to \IF}{\liminf}
	\frac{
		\mathcal{H}^{\mathrm{par}}_{\widehat{Y},T,L}
	}
	{T}
	\leq
	\underset{n\to \IF}{\limsup}\frac{
		\mathcal{H}^{\mathrm{par}}_{\widehat{Y},T,L}
	}{T}
	\leq 
	c_{\widehat{Y}}^{1/\kappa}
	\frac{
		\mathcal{H}^{\mathrm{par}}_{\kappa,c_{\widehat{Y}}^{1/\kappa}S,c_{\widehat{Y}}^{1/\kappa}L}
	}
	{c_{\widehat{Y}}^{1/\kappa} S}.
	\end{align*}
	Letting $S \to \infty$, we obtain 
	$\mathcal{H}^{\mathrm{par}}_{\widehat{Y},L}= c_{\widehat{Y}}^{1/\kappa}\mathcal{H}^{\mathrm{par}}_{\kappa,c_{\widehat{Y}}^{1/\kappa}L}$.
\end{proof}
\begin{proof}[Proof of Theorem \ref{mainth}]
	Consider two cases depending on $p=p(\alpha,\beta,\gamma,\kappa)$:
	
	\textbf{\underline{Case $p>0$.}} Fix $M'>0$ and $n>0$ and denote \\
	$$E_u \coloneq \left[\frac{1}{M'}u^{-2/\kappa+p},M' u^{-2/\kappa+p}\right],\,t_{u,n,k}\coloneq\frac{1}{M'}u^{-2/\kappa+p}+kn(ac_{\widehat{Y}})^{-1/\kappa}u^{-2/\kappa},$$
	$$ I_k(u,n)\coloneq[t_{u,n,k},\,t_{u,n,k+1}],$$
	where 
	$$0 \leq k \leq N(u,n),\quad   N(u,n)\coloneq\left[\frac{(M'-\frac{1}{M'})u^{p}
		(ac_{\widehat{Y}})^{1/\kappa}}{n}\right]-1,$$
	$$E(u,n) \coloneq \bigcup_{0 \leq k \leq N(u,n)}I_k(u,n), \quad K(u,n)\coloneq\{k \in \mathbb{Z}: 0 \leq k \leq N(u,n)\}.$$
	
	Denote
	\begin{align*}
	&
	\pi_u:=
	\pk*{\sup_{t \in [0,\delta]}
		\inf_{s \in [0,1]}	
		\left(
		\widehat{X}(t+\widehat{L}_{u,t}s)
		-d(t+\widehat{L}_{u,t}s)^{\hat{\gamma}}
		\right)
		>u
	},
	\\
	&
	\pi_{u,M'}:=
	\pk*{\sup_{t \in E_u}
		\inf_{s \in [0,1]}	
		\left(
		\widehat{X}(t+\widehat{L}_{u,t}s)
		-d(t+\widehat{L}_{u,t}s)^{\hat{\gamma}}
		\right)
		>u
	},
	\\
	&
	\Sigma_{u,n,M'}:=
	\sum_{k \in K(u,n)}
	\pk*{\sup_{t \in I_k(u,n)}
		\inf_{s \in [0,1]}	
		\left(
		\frac{
			\widehat{X}(t+\widehat{L}_{u,t}s)
		}{
			1+\frac{d(t+\widehat{L}_{u,t}s)^{\hat{\gamma}}}{u}
		}
		\right)
		>u
	},\\
	&
	\Sigma\Sigma_{u,n,M'}:=
	\sum_{ i \neq j, i,j\in K(u,n)} \pk*{ \sup_{t\in I_{i}(u,n)} \frac{\widehat{X}(t)}{1+\frac{dt^{\hat{\gamma}}}{u}}>u,
		\sup_{t\in I_j(u,n)} \frac{\widehat{X}(t)}{1+\frac{dt^{\hat{\gamma}}}{u}}>u}.
	\end{align*}
	Our plan for the first next steps consists in finding the asymptotics of the single sums $\Sigma_{u,n,M'}$ as $u \to \infty$, passing to limit in this asymptotics as $n \to \infty$ and checking that 
	\begin{align}
	\label{step1}
	\lim_{n \to \infty} \limsup_{u \to \infty}\frac{\left|\Sigma_{u,n,M'}- \pi_{u,M'} \right|}{\Psi(u)u^p} =0.
	\end{align}
	Next we will check
	\begin{align}
	\label{step2}
	\lim_{M' \to \infty} \limsup_{u \to \infty}\frac{\left|\pi_u- \pi_{u,M'} \right|}{\Psi(u)u^p} =0,
	\end{align}
	which is enough to finish the proof by \eqref{borell}.
	
	\underline{\emph{The negligibility of $E_u \backslash E(u,n)$}}
	Analogously to the proof of condition A1 from \citet*[Theorem 2.2]{novikov2025sojourn} we obtain
	\begin{align}
	\label{condA1}
	\Delta_{u,n,M'}:=\pk*{\underset{t \in [t_{u,n,N(u,n)},M'u^{-2/\kappa+p}]}{\sup}
		\frac{\widehat{X}(t)}{1+\frac{dt^{\hat{\gamma}}}{u}}
		>u}=O(\Psi(u))=o(\Psi(u)u^p)
	\end{align} 
	as $u \to \infty$.

	\underline{\emph{The negligibility of the double sums $\Sigma\Sigma_{u,n,M'}$}} By the same arguments as in the proof of condition A3 from
	\citet*[Theorem 2.2]{novikov2025sojourn} we obtain 
	\begin{align}
	\label{dsnegl}
	\lim_{n\rw\IF}\limsup_{u\rw\IF} \frac
	{\Sigma\Sigma_{u,n,M'}}{\Psi(u)u^p}= 0.
	\end{align}
	\underline{\emph{The asymptotics of the single sums $\Sigma_{u,n,M'}$}}
	
	Denote $t_k(\tau)\coloneq t_{u,n,k}+\tau n (a c_{\widehat{Y}}^{-1/\kappa}u^{-2/\kappa})$ and recall that (see \eqref{lup}) in the setting of \Cref{mainth} we have
	\[
	L_u=Lu^{-\frac{2+p(\alpha-\kappa)}{\alpha}}
	\]
	for the case $p>0$ and the abbreviation (see \eqref{lut}) 
	\[
	\widehat{L}_{u,t}=(t^{\kappa/\alpha}+L_u)^{\alpha/\kappa}-t.
	\]
	Remark that uniformly in $u>0$ and $t>0$ such that $\frac{L_u}{t^{\kappa/\alpha}} \to 0$ we have
	\\ 
	\begin{align}
	\label{Lht}
	\widehat{L}_{u,t}\sim \frac{\alpha}{\kappa}
	L_u t^{1-\kappa/\alpha}.
	\end{align}
	
	Therefore, setting (recall the notation $Z(t)=\overline{X}(t^{\kappa/\alpha})$)
	$$Z_{u,n,k}(\tau,s)\coloneq Z(t_k(\tau)+\widehat{L}_{u,t_k(\tau)}s), \quad \tau_{u,n,k}\coloneq t_{u,n,k}u^{2/\kappa-p}$$ 
	uniformly in $\tau_1,\tau_2, s_1,s_2 \in [0,1]$ we have
	\begin{align*}
	\widehat{L}_{u,t_k(\tau_1)}
	\sim\frac{\alpha}{\kappa}
	L_u t_k(\tau_1)^{1-\kappa/\alpha}
	\sim\frac{\alpha}{\kappa}
	Lu^{-\frac{2+p(\alpha-\kappa)}{\alpha}} t_{u,n,k}^{1-\kappa/\alpha}
	=\frac{\alpha}{\kappa}
	Lu^{-2/\kappa}\tau_{u,n,k}^{1-\kappa/\alpha}
	\end{align*}
	and (here we use \eqref{localBM})
	\begin{align}
	\label{calcvar}
	&u^2 \mathrm{Var}(Z_{u,n,k}(\tau_1,s_1)-Z_{u,n,k}(\tau_2,s_2))\notag
	\\&
	\sim u^2 aV_{\widehat{Y}}(t_k(\tau_1)+\widehat{L}_{u,t_k(\tau_1)}s_1,t_k(\tau_2)+\widehat{L}_{u,t_k(\tau_2)}s_2)\notag
	\\ &
	=ac_{\widehat{Y}} \left| 
	(\tau_1-\tau_2)n(ac_{\widehat{Y}})^{-1/\kappa}+(s_1-s_2)\frac{\alpha}{\kappa}L  \tau_{u,n,k}^{1-\kappa/\alpha}
	\right|^{\kappa}+o(1) \notag
	\\ &
	= \left| 
	(\tau_1-\tau_2)n+(s_1-s_2)(ac_{\widehat{Y}})^{1/\kappa}\frac{\alpha}{\kappa}L  \tau_{u,n,k}^{1-\kappa/\alpha}
	\right|^{\kappa}+o(1). 
	\end{align}
	
	Denote by
	$\widetilde{B}_{\kappa,v,w}(t,s),\,t,s \geq 0$ the centered continuous Gaussian process which satisfies $\widetilde{B}_{\kappa,v,w}(0,0)=0$ and
	\[
	\mathrm{Var}
	(\widetilde{B}_{\kappa,v,w}(t_1,s_1)-\widetilde{B}_{\kappa,v,w}(t_2,s_2))=|v(t_1-t_2)+w(s_1-s_2)|^{\kappa}.
	\]
	Then
	\begin{align}
	\label{parrescale}
	\mathcal{H}_{\widetilde{B}_{\kappa,v,w},1,1}^0
	=
	\mathcal{H}_{\widetilde{B}_{\kappa,1,1},v,w}^0
	=\mathcal{H}^{\mathrm{par}}_{\kappa,v,w}.
	\end{align}
	
	Under the notation of \Cref{mainth} denote $\sigma_{\widehat{X}}(t) \coloneq \mathrm{Var}(\widehat{X}(t))$. 
	
	If $n$ is large enough, then there exist 
	$$t_*,t_{**}\in [t_{u,n,k},t_{u,n,k+1}+\max\{\widehat{L}_{u,t_{u,n,k}},\widehat{L}_{u,t_{u,n,k+1}}\}]$$
	such that by \Cref{the-weak-conv-cor} applied to the Gaussian processes $Z_{u,n,k}(\tau,s)$ with $\tau,s\in[0,1]$ we obtain (\Cref{C0}, \Cref{C1} hold trivially, \Cref{C2} follows from \eqref{calcvar}, \Cref{C3} follows from (\ref{bddcorr}))
	\begin{align*}
	&\pk*{\sup_{t \in I_k(u,n)}
		\inf_{s \in [0,1]}	
		\left(
		\widehat{X}(t+\widehat{L}_{u,t}s)
		-d(t+\widehat{L}_{u,t}s)^{\hat{\gamma}}
		\right)
		>u
	}
	\\&
	=
	\pk*{
		\sup_{\tau \in [0,1]}\inf_{s \in [0,1]}
		\frac{
			\widehat{X}(t_k(\tau)+\widehat{L}_{u,t_k(\tau)}s)
		}{
			1+\frac{d(t_k(\tau)+\widehat{L}_{u,t_k(\tau)}s)^{\hat{\gamma}}}{u}
		}
		>u
	}
	\\&=
	\pk*{
		\sup_{\tau \in [0,1]} \inf_{s\in[0,1]}
		\left( Z(t_k(\tau)+\widehat{L}_{u,t_k(\tau)}s)
		\right)
		>u\left(1+\frac{dt_*^{\hat{\gamma}}}{u}\right)\sigma_{\widehat{X}}(t_{**})
	}
	\\&\sim\Psi\left(u\left(1+
	\frac{dt_{*}^{\hat{\gamma}}}{u}\right)
	\sigma_{\widehat{X}}(t_{**})\right)
	\mathcal{H}^{0}_{\widetilde{B}_{\kappa,n,(ac_{\widehat{Y}})^{1/\kappa}
			\frac{\alpha}{\kappa}L  \tau_{u,n,k}^{1-\kappa/\alpha}},1,1
	}
	\\&
	\sim\Psi(u) \exp \left(-d u^{1+(p-2 / \kappa) \hat{\gamma}} \tau_{u, n, k}^{\hat{\gamma}}-b u^{2+(p-2 / \kappa) \hat{\beta}} \tau_{u, n, k}^{\hat{\beta}}\right)
	\mathcal{H}^{\mathrm{par}}_{\kappa,n,(ac_{\widehat{Y}})^{1/\kappa}
		\frac{\alpha}{\kappa}L  \tau_{u,n,k}^{1-\kappa/\alpha}}
	\end{align*}
	uniformly in $k$ (in the final equivalence we use  \eqref{parrescale}). Consequently, we obtain 
	\begin{align*}
	&\Sigma_{u,n,M'}
	\\
	&\sim\sum_{k \in  K(u,n)} \Psi\left(u\right)
	\mathcal{H}^{\mathrm{par}}_{\kappa,n,(ac_{\widehat{Y}})^{1/\kappa}\frac{\alpha}{\kappa}L  \tau_{u,n,k}^{1-\kappa/\alpha}}\exp\left(-du^{1+(p-2/\kappa)\hat{\gamma}} \tau_{u,n,k}^{\hat{\gamma}}
	-bu^{2+(p-2/\kappa)\hat{\beta}} \tau_{u,n,k}^{\hat{\beta}}\right)
	\\
	&\sim
	\frac{u^p}{n(ac_{\widehat{Y}})^{-1/\kappa}}
	\Psi(u) 
	\\
	&\quad\times\sum_{k \in  K(u,n)} \mathcal{H}^{\mathrm{par}}_{\kappa,n,(ac_{\widehat{Y}})^{1/\kappa}\frac{\alpha}{\kappa}L  \tau_{u,n,k}^{1-\kappa/\alpha}} \exp\left(-du^{1+(p-2/\kappa)\hat{\gamma}} \tau_{u,n,k}^{\hat{\gamma}}
	-bu^{2+(p-2/\kappa)\hat{\beta}} \tau_{u,n,k}^{\hat{\beta}}\right)(\tau_{u,n,k+1}-\tau_{u,n,k})
	\\
	&\sim
	\frac{u^p}{n(ac_{\widehat{Y}})^{-1/\kappa}}
	\Psi(u) 
	\\
	&\quad\times\int\limits_{\frac{1}{M'}}^{M'} 
	\exp\left(-du^{1+(p-2/\kappa)\hat{\gamma}} z^{\hat{\gamma}}
	-bu^{2+(p-2/\kappa)\hat{\beta}} z^{\hat{\beta}}\right)\mathcal{H}^{\mathrm{par}}_{\kappa,n,(ac_{\widehat{Y}})^{1/\kappa}\frac{\alpha}{\kappa}L  z^{1-\kappa/\alpha}} dz
	\\&
	\sim
	\frac{u^p}{n(ac_{\widehat{Y}})^{-1/\kappa}}
	\Psi(u)
	\\&\quad
	\times
	\int\limits_{\frac{1}{M'}}^{M'} 
	\exp\left(-d\mathbb{I}(1+(p-2/\kappa)\hat{\gamma}=0) z^{\hat{\gamma}}
	-b\mathbb{I}(2+(p-2/\kappa)\hat{\beta}=0) z^{\hat{\beta}}\right)\mathcal{H}^{\mathrm{par}}_{\kappa,n,(ac_{\widehat{Y}})^{1/\kappa}\frac{\alpha}{\kappa}L  z^{1-\kappa/\alpha}} dz.
	\end{align*}
	In the last asymptotic equivalence we use the fact that $p \leq 2/\kappa-\max(2/\hat{\beta},1/\hat{\gamma})$.
	
	\underline{\emph{Passing to a limit in $n$ in $\Sigma_{u,n,M'}$; the proof of \eqref{step1}}} We have further 
	\begin{align}
	\label{sumwithm}
	&
	\lim_{n \to \infty} \limsup_{u \to \infty} \Bigg|
	\frac{\Sigma_{u,n}}{u^p\Psi(u)} 
	\notag
	\\&
	-
	(ac_{\widehat{Y}})^{1/\kappa}		
	\int\limits_{\frac{1}{M'}}^{M'} 
	\exp\left(-d\mathbb{I}(1+(p-2/\kappa)\hat{\gamma}=0) z^{\hat{\gamma}}
	-b\mathbb{I}(2+(p-2/\kappa)\hat{\beta}=0) z^{\hat{\beta}}\right)\mathcal{H}^{\mathrm{par}}_{\kappa,(ac_{\widehat{Y}})^{1/\kappa}\frac{\alpha}{\kappa}L  z^{1-\kappa/\alpha}} dz
	\Bigg|=0.
	\end{align}
	It is indeed possible to pass to the limit:
	firstly, the existence, positivity and finiteness of $\mathcal{H}^{\mathrm{par}}_{\kappa,L}$ was proved in \citet*{dkebicki2015parisianJAP}. In particular, $$\sup_{n > 1} \frac{\mathcal{H}^{\mathrm{par}}_{\kappa,n,0}}{n}
	<\infty.$$ Hence, we can swap the integral and the limit in $n$, since $$\sup_{n > 1, L \geq 0} \frac{\mathcal{H}^{\mathrm{par}}_{\kappa,n,L}}{n}\leq 
	\sup_{n > 1} \frac{\mathcal{H}^{\mathrm{par}}_{\kappa,n,0}}{n}
	<\infty.$$ Since as mentioned before for all $L \geq 0$ $\mathcal{H}^{\mathrm{par}}_L>0$, we have 
	\begin{align}
	\label{posconst}
	\int\limits_{\frac{1}{M'}}^{M'} 
	\exp\left(-d\mathbb{I}(1+(p-2/\kappa)\hat{\gamma}=0) z^{\hat{\gamma}}
	-b\mathbb{I}(2+(p-2/\kappa)\hat{\beta}=0) z^{\hat{\beta}}\right)\mathcal{H}^{\mathrm{par}}_{\kappa,(ac_{\widehat{Y}})^{1/\kappa}\frac{\alpha}{\kappa}L  z^{1-\kappa/\alpha}} dz>0.
	\end{align}
	
	Remark that by the Bonferroni inequality
	\begin{align*}
	\pi_{u,M'}
	&=&
	\pk*{\sup_{t \in E_u}
		\inf_{s \in [0,1]}	
		\left(
		\frac{
			\widehat{X}(t+\widehat{L}_{u,t}s)
		}{
			1+\frac{d(t+\widehat{L}_{u,t}s)^{\hat{\gamma}}}{u}
		}
		\right)
		>u
	}
	\\
	&\leq&
	\pk*{\sup_{t \in E(u,n)}
		\inf_{s \in [0,1]}	
		\left(
		\frac{
			\widehat{X}(t+\widehat{L}_{u,t}s)
		}{
			1+\frac{d(t+\widehat{L}_{u,t}s)^{\hat{\gamma}}}{u}
		}
		\right)
		>u
	}+\Delta_{u,n,M'}
	\\
	&\leq&
	\Sigma_{u,n,M'}+\Sigma\Sigma_{u,n,M'}+\Delta_{u,n,M'}.
	\end{align*}
	On the other hand, again by the Bonferroni inequality
	\begin{align*}
	\pi_{u,M'}
	&\geq&
	\pk*{\sup_{t \in E(u,n)}
		\inf_{s \in [0,1]}	
		\left(
		\frac{
			\widehat{X}(t+\widehat{L}_{u,t}s)
		}{
			1+\frac{d(t+\widehat{L}_{u,t}s)^{\hat{\gamma}}}{u}
		}
		\right)
		>u
	}
	\\
	&\geq&
	\Sigma_{u,n,M'}-\Sigma\Sigma_{u,n,M'}.
	\end{align*}
	
	Therefore, \eqref{condA1}, \eqref{dsnegl} and \eqref{sumwithm} give 
	
	\begin{align}
		\label{pum}
		\lefteqn{\pi_{u,M'}
		}\notag\\
		&\sim 
		u^p\Psi(u)
		(ac_{\widehat{Y}})^{1/\kappa}	\notag
		\\
		&\quad\times	
		\int\limits_{\frac{1}{M'}}^{M'} 
		\exp\left(-d\mathbb{I}(1+(p-2/\kappa)\hat{\gamma}=0) z^{\hat{\gamma}}
		-b\mathbb{I}(2+(p-2/\kappa)\hat{\beta}=0) z^{\hat{\beta}}\right)\mathcal{H}^{\mathrm{par}}_{\kappa,(ac_{\widehat{Y}})^{1/\kappa}\frac{\alpha}{\kappa}L  z^{1-\kappa/\alpha}} dz
	\end{align}
	for all $L \geq 0$.	
	
	\underline{\emph{The upper bound on $\pi_u$}}
	We have further 
	\begin{align*}
	\pi_u
	&\leq 
	\pk*{\sup_{t \in [0,\frac{1}{M'}u^{-2/\kappa+p}]}
		\inf_{s \in [0,1]}	
		\left(
		\widehat{X}(t+\widehat{L}_{u,t}s)
		-d(t+\widehat{L}_{u,t}s)^{\hat{\gamma}}
		\right)
		>u
	}
	\\&
	+\pi_{u,M'}
	\\&+
	\pk*{\sup_{t \in [M' u^{-2/\kappa+p},\delta]}
		\left(
		\inf_{s \in [0,1]}	
		\widehat{X}(t+\widehat{L}_{u,t}s)
		-d(t+\widehat{L}_{u,t}s)^{\hat{\gamma}}
		\right)
		>u
	}.
	\end{align*}
	Since for now we consider only the cases when $p>0$, by \eqref{largecut}, \Cref{smallcut} and \eqref{pum}
	for all $M'>0$ we obtain
	\begin{align*}
	\pi_u
	&\leq 
	(1+c(M')+o(1))
	u^p \Psi(u)(ac_{\widehat{Y}})^{1/\kappa}\\ 
	& \cEE{\times}\int\limits_{\frac{1}{M'}}^{M'} 
	\exp\left(-d\mathbb{I}(1+(p-2/\kappa)\hat{\gamma}=0) z^{\hat{\gamma}}
	-b\mathbb{I}(2+(p-2/\kappa)\hat{\beta}=0) z^{\hat{\beta}}\right)\mathcal{H}^{\mathrm{par}}_{\kappa,(ac_{\widehat{Y}})^{1/\kappa}\frac{\alpha}{\kappa}L  z^{1-\kappa/\alpha}} dz,
	\end{align*}
	where $c(M') \to 0$ as $M' \to \infty$. In particular we obtain
	\begin{align}
	\label{constbdd}
	\pi_u
	=O(u^p\Psi(u))
	\end{align}
	as $u \to \infty$. Therefore letting $M' \to \infty$, by the monotone convergence theorem we obtain
	\begin{align*}
	\lefteqn{
		\pk*{\sup_{t \in [0,\delta]}
			\inf_{s \in [0,1]}					\left(
			\widehat{X}(t+\widehat{L}_{u,t}s)
			-d(t+\widehat{L}_{u,t}s)^{\hat{\gamma}}
			\right)
			>u
		}
	}
	\\& \leq 
	(1+o(1))
	u^p\Psi(u)
	(ac_{\widehat{Y}})^{1/\kappa}	\\&
	\cEE{\times}
	\int\limits_{0}^{\infty} 
	\exp\left(-d\mathbb{I}(1+(p-2/\kappa)\hat{\gamma}=0) z^{\hat{\gamma}}
	-b\mathbb{I}(2+(p-2/\kappa)\hat{\beta}=0) z^{\hat{\beta}}\right)\mathcal{H}^{\mathrm{par}}_{\kappa,(ac_{\widehat{Y}})^{1/\kappa}\frac{\alpha}{\kappa}L  z^{1-\kappa/\alpha}} dz.
	\end{align*}
	\underline{\emph{The lower bound on $\pi_u$, proof of \eqref{step2},
			finiteness and positivity of the Parisian constant}}
	
	On the other hand
	for all $M'$
	\begin{align*}
	\pi_u
	&\geq 
	\pi_{u,M'}
	\\ &
	\geq  
	(1+o(1))
	u^p\Psi(u)
	(ac_{\widehat{Y}})^{1/\kappa}\\
	& \cEE{\times }		
	\int\limits_{0}^{\infty} 
	\exp\left(-d\mathbb{I}(1+(p-2/\kappa)\hat{\gamma}=0) z^{\hat{\gamma}}
	-b\mathbb{I}(2+(p-2/\kappa)\hat{\beta}=0) z^{\hat{\beta}}\right)\mathcal{H}^{\mathrm{par}}_{\kappa,(ac_{\widehat{Y}})^{1/\kappa}\frac{\alpha}{\kappa}L  z^{1-\kappa/\alpha}} dz.
	\end{align*}
	Consequently, if $p>0$, then by \eqref{borell}
	\begin{align}
	\label{finppos}
	\lefteqn{\pk*{\underset{t \in [0,T]}{\sup}\; \underset{s \in [0,L_u]}{\inf}\left( X(t+s)-d(t+s)^{\gamma}
			\right) > u }}\notag\\
	&=\pi_u+o(\Psi(u))\notag
	\\&
	\sim
	u^p\Psi(u)
	(ac_{\widehat{Y}})^{1/\kappa}		
	\int\limits_{0}^{\infty} 
	\exp\left(-d\mathbb{I}(1+(p-2/\kappa)\hat{\gamma}=0) z^{\hat{\gamma}}
	-b\mathbb{I}(2+(p-2/\kappa)\hat{\beta}=0) z^{\hat{\beta}}\right)\notag\\ &\cEE{\times}\mathcal{H}^{\mathrm{par}}_{\kappa,(ac_{\widehat{Y}})^{1/\kappa}\frac{\alpha}{\kappa}L  z^{1-\kappa/\alpha}} dz.
	\end{align}
	Remark that \eqref{constbdd} and \eqref{finppos} imply that 
	\begin{align*}
	\int\limits_{0}^{\infty} 
	\exp\left(-d\mathbb{I}(1+(p-2/\kappa)\hat{\gamma}=0) z^{\hat{\gamma}}
	-b\mathbb{I}(2+(p-2/\kappa)\hat{\beta}=0) z^{\hat{\beta}}\right)
	\mathcal{H}^{\mathrm{par}}_{\kappa,(ac_{\widehat{Y}})^{1/\kappa}\frac{\alpha}{\kappa}L  z^{1-\kappa/\alpha}} dz < \infty.
	\end{align*}
	Consequently, since we already know that the above integral is positive, and since in cases (i), (ii) (which correspond to $p>0$) of \Cref{mainth} one can check by a change of variables that
	\begin{align*}
	c=(ac_{\widehat{Y}})^{1/\kappa}\int\limits_{0}^{\infty} 
	\exp\left(-d\mathbb{I}(1+(p-2/\kappa)\hat{\gamma}=0) z^{\hat{\gamma}}
	-b\mathbb{I}(2+(p-2/\kappa)\hat{\beta}=0) z^{\hat{\beta}}\right)
	\mathcal{H}^{\mathrm{par}}_{\kappa,(ac_{\widehat{Y}})^{1/\kappa}\frac{\alpha}{\kappa}L  z^{1-\kappa/\alpha}} dz
	\end{align*}
	the proof for $p>0$ is finished.
	
	\textbf{\underline{Case $p=0$.}}
	In the cases when $p=0$, we will again proceed in several steps. As before, fix $M'>0$. 
	Let 
	$q \coloneq \min\left(-2/\kappa,-2/\hat{\beta},-1/\hat{\gamma}\right)$ and remark that in the setting of \Cref{mainth} for the case $p=0$ we have 
	\[
	L_u=Lu^{q(\kappa/\alpha)}.
	\] 
	
	Denote
	\[
	\pi'_u
	:=\pk*{\sup_{t \in [0,\delta]}
		\inf_{s \in [0,1]}	
		\widehat{X}(t+\widehat{L}_{u,t}s)
		-d(t+\widehat{L}_{u,t}s)^{\hat{\gamma}}
		>u
	}
	\]
	
	Observe that
	\begin{align}
	\label{fracrep}
	\lefteqn{\pk*{\sup_{t \in [0,M'u^q]}
			\inf_{s \in [0,1]}	
			\widehat{X}(t+\widehat{L}_{u,t}s)
			-d(t+\widehat{L}_{u,t}s)^{\hat{\gamma}}
			>u
	}}\notag
	\\
	&=
	\pk*{\sup_{\tau \in [0,M']}
		\inf_{s \in [0,1]}	
		\frac{
			\widehat{X}(u^q \tau+\widehat{L}_{u,u^q \tau}s)
		}{1+u^{-1}
			d(u^q \tau+\widehat{L}_{u,u^q \tau}s)^{\hat{\gamma}}
		}
		>u
	}=:\pi'_{u,M'}.
	\end{align}
	Our plan for the next steps will consist in identifying the asymtptotics of $\pi'_{u,M'}$, passing to the limit $M' \to \infty$ and checking
	\begin{align}
	\label{step2'}
	\lim_{M' \to \infty} \limsup_{u \to \infty}\frac{\left|\pi'_u- \pi'_{u,M'} \right|}{\Psi(u)} =0.
	\end{align}
	\underline{\emph{Asymptotics of $\pi'_{u,M}$}} We can apply  \Cref{the-weak-conv-cor} to the Gaussian process 
	\[
	Z_u(\tau,s)=\frac{
		\widehat{X}(u^q \tau+\widehat{L}_{u,u^q \tau}s)
	}{1+u^{-1}
		d(u^q \tau+\widehat{L}_{u,u^q \tau}s)^{\hat{\gamma}}
	}.
	\]
	
	Condition \Cref{C0} will hold obviously.
	Remark that 
	\begin{align}
	\label{Lht2}
	\widehat{L}_{u,u^q\tau}=((u^q\tau)^{\kappa/\alpha}+L_u)^{\alpha/\kappa}-u^q \tau=((u^q\tau)^{\kappa/\alpha}+L u^{q(\kappa/\alpha)})^{\alpha/\kappa}-u^q \tau=u^q((\tau^{\kappa/\alpha}+L)^{\alpha/\kappa}-\tau).
	\end{align}
	
	Let us check condition \Cref{C1}. 
	\begin{align*}
	\lefteqn{
		u^2\left(1-\mathrm{Var}\left(
		Z_u(\tau,s) 
		\right)\right)}
	\\
	&
	=
	u^2\left(1-\frac{\sigma_{\widehat{X}}(u^q \tau+\widehat{L}_{u,u^q \tau}s)}{1+u^{q\hat{\gamma}-1}
		d( \tau+u^{-q}\widehat{L}_{u,u^q \tau}s)^{\hat{\gamma}}}\right)
	\\
	&
	=
	o(1)+u^{1+q\hat{\gamma}}d(\tau+((\tau^{\kappa/\alpha}+L)^{\alpha/\kappa}-\tau)s)^{\hat{\gamma}}+u^{2+q\hat{\beta}}b(\tau+((\tau^{\kappa/\alpha}+L)^{\alpha/\kappa}-\tau)s)^{\hat{\beta}}
	\\ &\to
	\mathbb{I}\left(q=-1/\hat{\gamma}\right)d(\tau+((\tau^{\kappa/\alpha}+L)^{\alpha/\kappa}-\tau)s)^{\hat{\gamma}}
	+\mathbb{I}(q=-2/\hat{\beta})b(\tau+((\tau^{\kappa/\alpha}+L)^{\alpha/\kappa}-\tau)s)^{\hat{\beta}}
	\end{align*} 
	as $u \to \infty$, where for the last two asymptotics we used the fact that 
	\\$q \leq \min\left(-2/\hat{\beta},-1/\hat{\gamma}\right)$. 
	
	We show next \Cref{C2}. We have (recall the notation $Z(t)=\overline{X}(t^{\kappa/\alpha})$)
	\begin{align*}
	\lefteqn{u^2\left(\mathrm{Var}(
		\overline{Z}_u(\tau_1,s_1)-\overline{Z}_u(\tau_2,s_2))\right)}
	\\&=u^2 \mathrm{Var}(Z(u^q \tau_1+\widehat{L}_{u,u^q \tau_1}s_1)-Z(u^q \tau_2+\widehat{L}_{u,u^q \tau_2}s_2))
	\\&\sim
	2aV_{\widehat{Y}}( \tau_1+u^{-q}\widehat{L}_{u,u^q \tau_1}s_1, \tau_2+u^{-q}\widehat{L}_{u,u^q \tau_2}s_2)u^{2+q \kappa}
	\\&\to
	2aV_{\widehat{Y}}( \tau_1+((\tau_1^{\kappa/\alpha}+L)^{\alpha/\kappa}-\tau_1)s_1, \tau_2+((\tau_2^{\kappa/\alpha}+L)^{\alpha/\kappa}-\tau_2)s_2)\mathbb{I}(q=-2/\kappa)
	\end{align*} 
	as $u \to \infty$, where for the last
	asymptotics we used \eqref{Lht2} and the fact that \\$q \leq \min\left(-2/\hat{\beta},-1/\hat{\gamma}\right)$. Condition \Cref{C3}, as before, follows from (\ref{bddcorr}).
	
	Remark that
	\begin{align*}
	&\mathbb{I}(q=-1/\hat{\gamma})
	=\mathbb{I}(2 \gamma \leq \min(\alpha,\beta)),
	\\
	&	
	\mathbb{I}(q=-2/\hat{\beta})
	=\mathbb{I}(\beta \leq \min(\alpha,2 \gamma)),
	\\
	&
	\mathbb{I}(q=-2/\kappa)
	=\mathbb{I}(\alpha\leq \min(\beta,2 \gamma)).	
	\end{align*}
	
	Therefore, setting 
	\[
	f(\tau,s)=
	\tau+((\tau^{\kappa/\alpha}+L)^{\alpha/\kappa}-\tau)s
	\] 
	and
	\begin{align*}
	\hat{h}(\tau,s)&\coloneq&\mathbb{I}\left(2\gamma\leq \min\left(\alpha,\beta\right)\right)d  f(\tau,s)^{\hat{\gamma}}
	+\mathbb{I}(\beta\leq\min\left(\alpha,2\gamma\right))b  f(\tau,s)^{\hat{\beta}},
	\\
	\eta(\tau,s)&\coloneq&\widehat{Y}(a^{1/\kappa}f(\tau,s))\mathbb{I}(\alpha\leq\min\left(\beta,2\gamma\right))
	\end{align*} we obtain
	\begin{align}
	\label{zerasymp}
	\pi'_{u,M'}
	\sim\Psi(u)\mathcal{H}^{\hat{h}}_{\eta,M',1}
	\end{align}
	as $u \to \infty$.
	
	Remark that $f(\tau,s)=f(f(\tau,s),0)$, and therefore $\hat{h}(\tau,s)=\hat{h}(f(\tau,s),0)$, $\eta(\tau,s)=\eta(f(\tau,s),0)$.
	
	Therefore, if $\alpha \leq \min\left(\beta,2\gamma\right)$, then 
	\begin{align*}
	&\sup_{\tau\in [0,M']} \inf_{s \in [0,1]}
	\left( \sqrt{2}\eta(\tau,s)-\mathrm{Var}(\eta(\tau,s))-\hat{h}(\tau,s)
	\right)
	\\
	&=
	\sup_{\tau\in [0,M']} \inf_{s \in [0,1]}
	\left( \sqrt{2}\eta(f(\tau,s),0)-\mathrm{Var}(\eta(f(\tau,s),0))-\hat{h}(f(\tau,s),0)
	\right)
	\\
	&=
	\sup_{\tau\in [0,M']}
	\inf_{s \in [f(\tau,0),f(\tau,1)]}
	\left(
	\sqrt{2}\eta(s,0)-\mathrm{Var}(\eta(s,0))-\hat{h}(s,0)
	\right)
	\\
	&=
	\sup_{\tau\in [0,M']}
	\inf_{s \in [\tau,(\tau^{\kappa/\alpha}+L)^{\alpha/\kappa}]}
	\left(
	\sqrt{2}
	\widehat{Y}(a^{1/\kappa}s)-\mathrm{Var}(\widehat{Y}(a^{1/\kappa}s))-\hat{h}(s,0)
	\right)
	\\&=
	\sup_{\tau\in [0,M']}
	\inf_{s \in [a^{1/\kappa}\tau,a^{1/\kappa}(\tau^{\kappa/\alpha}+L)^{\alpha/\kappa}]}
	\left(
	\sqrt{2}\widehat{Y}(s)-\mathrm{Var}(\widehat{Y}(s))-\hat{h}(a^{-1/\kappa}s,0)
	\right)
	\\&=
	\sup_{\tau\in [0,M']}
	\inf_{s \in [a^{1/\alpha}\tau^{\kappa/\alpha},a^{1/\alpha}(\tau^{\kappa/\alpha}+L)]}
	\left(
	\sqrt{2}Y(s)-\mathrm{Var}(Y(s))-\hat{h}(a^{-1/\kappa}s^{\alpha/\kappa},0)
	\right)\\
	&=
	\sup_{t\in [0,a^{1/\alpha}M'^{\kappa/\alpha}]}
	\inf_{s \in [0,a^{1/\alpha}L]}
	\left(
	\sqrt{2}Y(t+s)-\mathrm{Var}(Y(t+s))-\hat{h}(a^{-1/\kappa}(t+s)^{\alpha/\kappa},0)
	\right).
	\end{align*}
	Therefore, we obtain 
	\begin{align}
	\label{pzerowithm}
	\pi'_{u,M'}\sim\Psi(u)
	\mathcal{H}^{\mathrm{par},h}_{Y\mathbb{I}(\alpha\leq\min\left(\beta,2\gamma\right)),a^{1/\alpha}M'^{\kappa/\alpha},a^{1/\alpha}L},
	\end{align}
	with $$h(t)=a^{-\beta/\alpha}bt^{\beta}\mathbb{I}(\beta\leq\min\left(\alpha,2\gamma\right))+a^{-\gamma/\alpha}dt^{\gamma}\mathbb{I}\left(2\gamma\leq \min\left(\alpha,\beta\right)\right)
	,$$
	since 
	\[
	\hat{h}(\tau,0)
	=\mathbb{I}\left(2\gamma\leq \min\left(\alpha,\beta\right)\right)d  \tau^{\hat{\gamma}}
	+\mathbb{I}(\beta\leq\min\left(\alpha,2\gamma\right))b  \tau^{\hat{\beta}},
	\]
	so
	\[
	\hat{h}(a^{-1/\kappa}(t+s)^{\alpha/\kappa},0)
	=h(t+s).
	\]
	
	\underline{\emph{The estimate of $\pi'_u$, proof of \eqref{step2'},  finiteness and positivity of the Parisian constant}}
	
	By \eqref{fracrep}  
	\[
	0 \leq \pi'_u-\pi'_{u,M}
	\leq
	\pk*{\sup_{t \in [M'u^q,\delta]}
		\inf_{s \in [0,1]}	
		\left(
		\widehat{X}(t+\widehat{L}_{u,t}s)
		-d(t+\widehat{L}_{u,t}s)^{\hat{\gamma}}
		\right)
		>u
	},
	\]
	so \eqref{step2'} follows from \eqref{largecut}.
	In particular, picking $M'$ large enough, as $u \to \infty$, from \eqref{pzerowithm} we obtain
	\begin{align}
	\label{constbdd'}
	\pi'_u
	=O(\Psi(u)).
	\end{align}
	Therefore, letting $M' \to \infty$ in \eqref{pzerowithm} we obtain
	\[
	\pi'_u 
	\sim 
	\Psi(u)
	\mathcal{H}^{\mathrm{par},h}_{Y\mathbb{I}(\alpha\leq\min\left(\beta,2\gamma\right)),\ksn{\infty},a^{1/\alpha}L}
	\]
	with as $u \to \infty$, and $\mathcal{H}^{\mathrm{par},h}_{Y\mathbb{I}(\alpha\leq\min\left(\beta,2\gamma\right)),\ksn{\infty},a^{1/\alpha}L}<\infty$.
	
	By \eqref{borell} this implies
	\[
	\pk*{\underset{t \in [0,T]}{\sup}\; \underset{s \in [0,L_u]}{\inf}\left( X(t+s)-d(t+s)^{\gamma}
		\right) > u
	}	
	\sim 
	\Psi(u)
	\mathcal{H}^{\mathrm{par},h}_{Y\mathbb{I}(\alpha\leq\min\left(\beta,2\gamma\right)),\ksn{\infty},a^{1/\alpha}L}		
	\]
	as desired.
	
	Since clearly $\mathcal{H}^{\mathrm{par},h}_{Y\mathbb{I}(\alpha\leq\min\left(\beta,2\gamma\right)),\ksn{\infty},a^{1/\alpha}L}>0$, the proof of the case $p=0$ of \Cref{mainth} is finished.		
\end{proof}

\section*{Acknowledgements}
Financial support by SNSF Grant 200021-196888 is kindly acknowledged.

\section*{Disclosure statement}
The author has no conflicts of interest to declare that are relevant to the content of this article.

\section*{Funding}
This work was supported by the SNSF under Grant 200021-196888.

\bibliographystyle{apacite}
\bibliography{EEEA}



\end{document}